\documentclass[a4paper,11pt]{article}
\usepackage[utf8]{inputenc}
\usepackage[T1]{fontenc}
\usepackage[english]{babel}
\usepackage{a4wide}

\usepackage[]{hyperref}
\hypersetup{
    colorlinks=true,       
    linkcolor=red,          
    citecolor=blue,        
    filecolor=magenta,      
    urlcolor=cyan           
}

\usepackage[leqno]{amsmath}
\usepackage{amssymb}
\usepackage{mathrsfs}
\usepackage{amsthm}
\usepackage{amsxtra}
\usepackage{bm}
\usepackage[titletoc]{appendix}

\usepackage{graphicx}

\theoremstyle{plain}
\newtheorem{thm}{Theorem}[section]
\newtheorem{prop}[thm]{Proposition}
\newtheorem{lem}[thm]{Lemma}
\newtheorem{cor}[thm]{Corollary}

\theoremstyle{definition}
\newtheorem{defn}[thm]{Definition}

\newtheorem{rem}{Remark}

\numberwithin{equation}{section}

\title{Strichartz estimates and local existence for the capillary water waves with non-Lipschitz initial velocity}
\author{
Thibault de Poyferr\'{e}\footnote{UMR 8553 du CNRS, Laboratoire de Mathématiques et Applications de l'Ecole Normale Supérieure, 75005 Paris, France. Email: tdepoyfe@dma.ens.fr}
\and
Quang Huy Nguyen
\footnote{UMR 8628 du CNRS, Laboratoire de Math\'ematiques d'Orsay, Universit\'e Paris-Sud, 91405 Orsay Cedex, France. Email: quang-huy.nguyen@math.u-psud.fr }
}
\date{}

\DeclareMathOperator{\dist}{dist}
\DeclareMathOperator{\RE}{Re}

\DeclareMathOperator{\supp}{supp}
\DeclareMathOperator{\dive}{div}
\DeclareMathOperator{\Hess}{Hess}
\DeclareMathOperator{\Op}{Op}

\def \div {\dive}
\def\C{\mathcal{C}}
\def\d{\,\mathrm{d}}
\def\eps{\varepsilon}
\def\F{\mathcal{F}}

\def\la{\left\vert}
\def\lA{\left\Vert}
\def\lb{\left[}
\def\lB{\left\{}
\def\lp{\left(}
\def\ls{\left\langle}
\def\les{\lesssim}
\def\mez{\frac{1}{2}}
\def\tmez{\frac{3}{2}}
\def\N{\bm{\mathrm{N}}}
\def\ph{\varphi}
\def\R{\bm{\mathrm{R}}}
\def\ra{\right\vert}
\def\rA{\right\Vert}
\def\rb{\right]}
\def\rB{\right\}}
\def\rp{\right)}
\def\rs{\right\rangle}

\def\eps{\varepsilon}
\def\mez{\frac{1}{2}}
\def\tdm{\frac{3}{2}}
\newcommand{\bq}{\begin{equation}}
\newcommand{\eq}{\end{equation}}
\newcommand{\bqa}{\begin{eqnarray*}}
\newcommand{\eqa}{\end{eqnarray*}}
\newcommand{\hk}{\hspace*{.15in}}
\def\xC{\mathbf{C}}
\def\xN{\mathbf{N}}
\def\xR{\mathbf{R}}

\def\xZ{\mathbf{Z}}
\begin{document}
\maketitle
\abstract{ 
We consider the gravity-capillary waves in any dimension and in fluid domains with general bottoms. Using the paradiferential reduction established in~\cite{NgPo}, we prove Strichartz estimates for solutions to this problem, at a low regularity level such that initially, the velocity field can be  non-Lipschitz up to the free surface. We then use those estimates to solve the Cauchy problem at this level of regularity.
}
\section{Introduction}
\subsection{Equations}

The water waves problem is the study of the motion of an incompressible inviscid fluid, lying above a fixed bottom and below an atmosphere, from which it is separated by a free surface. At equilibrium, this surface is flat. As soon as one perturbs this equilibrium, the surface will be put in motion by the combined action of gravity and surface tension.

The  velocity of such a fluid will obey the classical Euler equations of fluid dynamics, with the added difficulty of the moving surface.
As such, the domain occupied by the fluid will depend on the time at which it is observed. We thus consider the time-dependent domain  
$$
\Omega = \{(t,x,y) \in[0,T] \times \R^d \times \R:(x, y)\in \Omega_t\} 
$$
where each $\Omega_t$ is a domain located underneath a free surface 
$$
\Sigma_t = \{(x,y)  \times \R^d \times \R: y=\eta(t, x)\} 
$$
and above a fixed bottom $\Gamma=\partial\Omega_t\setminus \Sigma_t$. The physical dimensions are $d=1,~2$. We make the following important assumption on the domain: {\it Assumption $(H_t)$\\
$\Omega_t$ is the intersection of the half space 
\[
\Omega_{1,t}= \{(x,y)  \times \R^d \times \R: y<\eta(t, x)\} 
\]
and an open connected set $\Omega_2$ containing a fixed strip around $\Sigma_t$, i.e., there exists $h>0$ such that 
\[
 \{(x,y)  \times \R^d \times \R: \eta(t, x)-h\le y\le\eta(t, x)\} \subset \Omega_2.
\]
}
This important hypothesis prevents the bottom from emerging, or even from coming arbitrarily close to the free surface. The study of water waves without it is an open problem.

It is customary in mathematics to simplify the problem further by supposing the motion of the fluid to be irotational. This covers a large class of physical applications. Now under this additional hypothesis, and if the domain is simply connected, the velocity field $v$ admits a potential $\phi:\Omega \to \R$, i.e, $v=\nabla \phi$. An important observation by Zakharov \cite{Zack} is that the motion is then completely determined by the value of the elevation~$\eta(t,x)$ and of the trace~$\psi(t,x)=\phi(t,x,\eta(t,x))$ of the potential at the surface. We can then find~$\phi$ as the  unique variational solution of 
\bq\label{phi}
\Delta\phi =0\text{~in}~\Omega_t,\quad \phi(t, x, \eta(t, x))=\psi(t, x),\quad\partial_n\phi\rvert_{\Gamma}=0
\eq
Now following Craig and Sulem \cite{CrSu} to write a compact version of the equations, we introduce  the Dirichlet-Neumann operator
\begin{align*}
G(\eta) \psi &= \sqrt{1 + \vert \nabla_x \eta \vert ^2}
\Big( \frac{\partial \phi}{\partial n} \Big \arrowvert_{\Sigma}\Big)\\
&= (\partial_y \phi)(t,x,\eta(t,x)) - \nabla_x \eta(t,x) \cdot(\nabla_x \phi)(t,x,\eta(t,x)).
\end{align*}

The water wave system can now be rewritten as the following so-called Zakharov-Craig-Sulem system on~$(\eta,\psi)$ : 
\begin{equation}\label{eq:Zak}
\left\{
\begin{aligned}
&\partial_t \eta = G(\eta) \psi,\\
&\partial_t \psi + g\eta-H(\eta)+\mez \vert \nabla_x \psi \vert^2 - \mez \frac{(\nabla_x \eta \cdot \nabla_x \psi + G(\eta)\psi)^2}{1+ \vert \nabla_x \eta \vert^2}=0,
\end{aligned}
\right.
\end{equation}
where $H(\eta)$ is the mean curvature of the free surface:
\[
H(\eta)=\div\left( \frac{\nabla\eta}{\sqrt{1+|\nabla\eta|^2}}\right).
\]
The vertical and horizontal components of the velocity will play an important role in the analysis of system \eqref{eq:Zak}. These quantities  can be expressed in terms of $\eta$ and $\psi$ as
\begin{equation}\label{BV}
B = (v_y)\arrowvert_\Sigma = \frac{ \nabla_x \eta \cdot \nabla_x \psi + G(\eta)\psi} {1+ \vert \nabla_x \eta \vert^2},\quad
V= (v_x)\arrowvert_\Sigma  =\nabla_x \psi - B \nabla_x \eta.
 \end{equation}

\subsection{The problem}
In the present paper and its companion~\cite{NgPo}, we  aim  to prove local existence for rough data below the energy threshold, using the dispersive properties of this system. The local existence of solutions for the water waves system has been extensively studied by many authors, among them Nalimov \cite{Nalimov}, Yosihara \cite{Yosihara},  Coutand-Shkoller \cite{CouShk}, Craig \cite{Craig}, Wu \cite{Wu1,Wu2}, Christodoulou-Lindblad~\cite{CL}, Lindblad\cite{Lind}, Lannes~\cite{Lannes},  Ming-Zhang \cite{MiZh} and for the case with surface tension, in Beyer-G\"{u}nther~\cite{BG}, Ambrose-Masmoudi~\cite{AM1,AM2}, Shatah-Zeng \cite{SZ1, SZ2, SZ3}.
For the full  system with gravity and surface tension, in terms of regularity of data the result of Alazard, Burq and Zuily \cite{ABZ1} reaches an important level:
\bq\label{regu:ABZ1}
(\eta_0, \psi_0)\in H^{s+\mez}(\R^d)\times H^s(\R^d),\quad s>2+\frac d2.
\eq
Observe that by the formulas \eqref{BV}, this is the optimal Sobolev index to ensure that the initial velocity  field is Lipschitz up to the free surface, which is a quite natural criterion for the flow of fluid particles to be well-defined, in terms of the Cauchy-Lipschitz theorem.\\
Now, let us look at the linearized around the rest state $(\eta=0, \psi=0)$ of \eqref{eq:Zak}, with~$g=0$. It reads  
$$\partial_t\Phi+i\la D\ra^\tmez\Phi=0,$$
where $\Phi=\la D\ra^\mez\eta+i\psi$. This linear equation is dispersive (see paragraph \ref{mainresult} below), and we expect the full system to exhibit dispersive properties as well. The consequences of this dispersion for long time dynamics have been extensively studied in recent years, starting from the works of Wu~\cite{Wu2D,Wu3D}, by Germain-Masmoudi-Shatah~\cite{GMS1,GMS2}, Alazard-Delort~\cite{AD1,AD2}, Ionescu-Pusateri~\cite{IP1,IP2}, Hunter-Ifrim-Tataru~\cite{HIT}, and Ifrim-Tataru~\cite{IT1,IT2}. 

In this paper, we are interested in the consequences for short time and rough data, the so-called Strichartz estimates. They are a family of local in time estimates improving the Sobolev inequalities for a solution of the system, which can then be used to improve the energy estimates and thus lead to well-posedness with less regularity for initial data. The method to obtain such results for quasi-linear wave equations was developed by Bahouri and Chemin \cite{BaCh-AJM} and by Tataru, notably in \cite{TataruNS}. 

  However, little is known about Strichartz estimates  for water waves systems. In~\cite{CHS}, Christianson-Hur-Staffilani  proved Strichartz estimate for 2D gravity-capillary waves under another formulation. Then, Alazard-Burq-Zuily  obtained in~\cite{ABZ2} such a result for solutions to \eqref{eq:Zak} at regularity \eqref{regu:ABZ1}. We want to improve this in two ways, by proving Strichartz estimates :
\begin{itemize}
\item[(1)] valid for 3D waves,
\item[(2)] that can be used to improve the threshold \eqref{regu:ABZ1}, for both 2D and 3D waves.
\end{itemize}
 In fact, the method used in \cite{ABZ2} relies on a reduction specific to the dimension $d=1$, so for $(1)$ we need another method. On the other hand, for $(2)$ one need to derive the Strichartz estimates assuming that the solution is less regular than \eqref{regu:ABZ1} and consequently, the coefficients appearing in the equation are rougher. Such a program has been carried out by Alazard, Burq and Zuily in \cite{ABZ} for the pure gravity case. In fact, we shall follow here a similar approach, that is, proving dispersive estimates using semiclassical analysis. The main novelty is that here, the equation has infinite speed of propagation, so that we need to construct a parametrix in semiclassical time. Also, we use at the fullest the regularity of the coefficients to expand the lifespan of this parametrix. \\
\hk The first step in this program is to reduce system \eqref{eq:Zak} to a single equation to which the method for quasilinear equations can be applied. This uses paradifferential calculus, whose notations and main features are recalled in Appendix \ref{ap:para}. Specifically, we proved in the companion paper \cite{NgPo} that
assuming~$(\eta,\psi)$ to be a solution of~\eqref{eq:Zak}  satisfying condition $(H_t)$ for all times~$t\in I=[0, T]$, such that 
\begin{equation*}
	(\eta,\psi)\in L^\infty(I;H^{s+\mez}(\R^d)\times H^s(\R^d))\cap L^p(I;W^{r+\mez,\infty}(\R^d)\times W^{r,\infty}(\R^d)),
\end{equation*}
where 
\bq\label{mixednorm}
s>\frac{3}{2}+\frac{d}{2},~2<r<s-\frac d2+\mez,
\eq
the system \eqref{eq:Zak} can be rewritten as
\bq\label{reduced}
\partial_tu+ T_V\cdot\nabla  u+iT_{\gamma+\omega} u=f,
\eq
where the principal symbol $\gamma$ is of order $3/2$, real-valued;  the sub-principal symbol $\omega$ is of order $1/2$, complex-valued; the transport field $V$ is the horizontal part of the velocity field at the free surface: $V= (v_x)\arrowvert_\Sigma$ and the remainder term $f$ satisfies the following {\it tame estimate} for a.e. $t\in I$
\[
		\lA f(t)\rA_{H^s}\leq\F\lp\lA\eta(t)\rA_{H^{s+\mez}},\lA\psi(t)\rA_{H^s}\rp\lb1+\lA\eta(t)\rA_{W^{r+\mez,\infty}}+\lA\psi(t)\rA_{W^{r,\infty}}\rb.
\]
  In this article, we shall study equation \eqref{reduced} independently from its origin in the water waves system, proving a priori Strichartz estimates for its solution. This will  imply a priori Strichartz estimates for the gravity-capillary waves system \eqref{eq:Zak}. We will then combine them with the energy and contraction estimates and with a blow-up criterion, all proved in~\cite{NgPo}, to solve the Cauchy problem at low regularity such that the initial velocity field may fail to be Lipschitz (up to the surface).

\subsection{ Main results}\label{mainresult}
\hk Remark that the linearized system of \eqref{eq:Zak} around the rest state~$(\eta=0, \psi=0)$ when~$g=0$ reads
\[
\begin{cases}
\partial_t\eta-|D_x|\psi=0,\\
\partial_t\psi-\Delta \eta=0,
\end{cases}
\]
which can be written 
\[
\partial_t\Phi +i|D_x|^\tdm \Phi=0,\quad\text{with}~\Phi=|D_x|^\mez\eta+i\psi.
\]
It classically follows from the explicit formula for the solution, Litlewood-Paley decomposition, stationary phase and a TT* argument that
\[
\lA \Phi\rA_{L^pW^{s-\frac{d}{2}-\mu_{opt}, \infty}}\les \lA \Phi\arrowvert_{t=0}\rA_{H^s}
\]
with 
\bq\label{Str:lin}
\begin{cases}
\mu_{opt}=\frac{3}{8},~p=4\quad\text{if}~d=1, \\
\mu_{opt}=\frac{3}{4}-\eps,~p=2\quad\text{if}~d\ge 2.
\end{cases}
\eq
Our first result states that the fully nonlinear gravity-capillary waves system \eqref{eq:Zak} satisfies a similar estimate to \eqref{Str:lin} with 
\[
\mu=\frac{3}{20}-\eps~\text{if}~d=1, \quad \mu=\frac{3}{10}-\eps~\text{if}~d\ge 2.
\]
More precisely, we prove
\begin{thm}\label{main:theo}
	Let~$I=[0,T]$, $d\geq1$ and $p\ge 1$. Let $V(t, x):I\times \xR^d\to \xR^d$ be a vector field and $\gamma,~\omega,~\omega_1$ be the symbols defined by \eqref{def:gamma}. Consider 
\begin{equation}	\label{eq:mu}
\mu\in \lp0, \frac{3}{20}\rp,\quad p=4\quad \text{if}~d=1;\quad\mu\in \lp0,  \frac{3}{10}\rp,\quad p=2\quad \text{if}~d\ge 2.
\end{equation}
 Then for any $\sigma\in\R$ there exist~$k=k(d)\in\N$ and~$\F:\R^+\rightarrow\R^+$ non-decreasing such that the following property holds:\\
\hk  if $f\in L^p( I;H^{\sigma-\frac{9}{10}}(\R^d))$ and $u\in L^\infty(I;H^{\sigma}(\R^d))$ satisfy
	\begin{equation*}
		\lp\partial_t+T_V\cdot\nabla+iT_{\gamma+\omega}\rp u=f,
	\end{equation*}
	then we have
	\begin{equation}\label{main:Strichartz}
		\lA u\rA_{L^p(I;W^{\sigma-\frac{d}{2}+\mu}(\R^d))}\leq\F(\Upsilon)\lb\lA f\rA_{L^p( I;H^{\sigma-\frac{9}{10}}(\R^d))}+\lA u\rA_{L^\infty( I;H^{\sigma}(\R^d))}\rb,
	\end{equation}
where $\Upsilon$ is the sum of semi-norms of the coefficients, defined by \eqref{N_kgamma} and \eqref{M_kgamma}:
\[
\Upsilon=\lA V\rA_{L^p\lp I, W^{1,\infty}(\R^d)\rp}+\mathcal{M}_{k}(\gamma)(I)+\mathcal{N}_{k}(\gamma)(I)+\mathcal{N}_{k}(\omega)(I).
\]
\end{thm}

As a corollary, this will imply the corresponding Strichartz estimate for the water waves equation.   To be concise in the following statements let us define the quantities that control the system:
\begin{align*}
&\text{Sobolev norms}:~M_{\sigma,T}=\Vert (\eta, \psi)\Vert_{L^{\infty}([0, T]; H^{\sigma+\mez}\times H^\sigma)},~ M_{\sigma,0}=\Vert (\eta_0, \psi_0)\Vert_{H^{\sigma+\mez}\times H^\sigma},\\
&\text{Blow-up norm}: N_{\sigma,T}=\Vert (\eta, \psi)\Vert_{L^1([0, T]; W^{\sigma+\mez, \infty}\times W^{\sigma, \infty})},\\
& \text{Strichartz norm}:Z_{\sigma,T}=\Vert (\eta, \psi)\Vert_{L^p([0, T]; W^{\sigma+\mez, \infty}\times W^{\sigma, \infty})}.
\end{align*}
\begin{cor}	\label{cor:Stri}
	Let~$d\geq1$, $h>0$ and~$(s,r)\in\R^2$ such that
	\[s>2+\frac d2-\mu,\quad 2<r<s-\frac d2+\mu,\]
	with~$\mu$  and $p$ as in~\eqref{eq:mu}.
	Then there exists a non-decreasing~$\F:\R^+\rightarrow\R^+$ such that for all~$T\in(0,1]$, for all~$(\eta,\psi)$ smooth solution of~\eqref{eq:Zak} on~$[0,T]$ satisfying $\inf_{t\in [0, T]}\dist(\eta(t), \Gamma)>h,$
 there holds
	\[Z_r(T)\leq\F\lp T\F\lp M_s(T)+Z_r(T)\rp\rp.\]
\end{cor}
In~\cite{NgPo}, we have established the following energy estimate, blow up criterion and contraction estimate.
\begin{prop}[\protect{\cite[Theorem~1.1]{NgPo}}] \label{prop:aprioriestimate}
Let~$d\geq 1$, $h>0,~p> 1$ and $(r, s)\in \xR^2$ such that
\bq\label{intro:reg}
 s>\tdm+\frac{d}{2},\quad2<r<s+\frac12-\frac d2.
\eq
 Then there exists a non-negative, non-decreasing function~$\F$  such that: for all $T\in (0, 1]$ and all $(\eta, \psi)$ smooth solution to \eqref{eq:Zak} on $[0, T]$ satisfying $\inf_{t\in [0, T]}\dist(\eta(t), \Gamma)>h,$
 there holds
\[
 M_{s, T}\leq\F\lp M_{s,0}+T\F\lp M_{s, T}+Z_{r,T}\rp\rp.
\]
\end{prop}
\begin{prop}	[\protect{\cite[Theorem~1.2]{NgPo}}]\label{theo:blowup}
		Let~$d\geq 1,~h>0$ and indices
	$$\frac32+\frac d2<s_0<s-\mez,\quad2<r<s_0+\frac12-\frac d2.$$
Let $T^*=T^*(\eta_0, \psi_0, h)$ be the maximal time of existence and
\[
(\eta,\psi)\in L^{\infty}\lp[0, T^*); H^{s+\mez}\times H^s\rp
\]
	be the maximal solution of~\eqref{eq:Zak} with prescribed data $(\eta_0, \psi_0)$ satisfying $\dist(\eta_0, \Gamma)>h.$
	Then if ~$T^*$ is finite, we have 
\[
\limsup_{T\rightarrow T^*}\lp M_{s_0}(T)+N_r(T)\rp=+\infty.
\]
\end{prop}
\begin{prop}[\protect{\cite[Theorem~5.9]{NgPo}}]\label{theo:contraction}
Let $p$ and $\mu$ as in \eqref{eq:mu}. Let $(\eta_j, \psi_j)$,~$j=1,2$ be two solutions to \eqref{eq:Zak} on $I=[0, T],~0<T\le 1$ such that 
\[
(\eta_j, \psi_j)\in L^\infty(I; H^{s+\mez}(\xR^d)\times H^s(\xR^d))\cap L^p(I; W^{r+\mez}(\xR^d)\times W^{r,\infty}(\xR^d))
\]
with 
\[
s>\tmez+\frac d2,~\quad 2<r<s-\frac d2+\mu;
\]
such that $\inf_{t\in [0, T]}\dist(\eta_j(t), \Gamma)>h>0.$ Set
\[
M^j_{s,T}:=\Vert (\eta_j, \psi_j)\Vert_{L^{\infty}([0, T]; H^{s+\mez}\times H^s)}, \quad  Z^j_{r,T}:=\Vert (\eta_j, \psi_j)\Vert_{L^p([0, T]; W^{r+\mez, \infty}\times W^{r, \infty})}.
\]
Consider the differences $\delta\eta:=\eta_1-\eta_2,~ \delta\psi:=\psi_1-\psi_2$ and their norms in Sobolev space and H\"older space:
\[
P_{T}:=\lA (\delta\eta, \delta\psi)\rA_{L^\infty(I; H^{s-1}\times H^{s-\tmez})}+\lA (\delta\eta, \delta\psi)\rA_{L^p(I; W^{r-1,\infty}\times W^{r-\tmez,\infty})}.
\]
Then there exists a non-decreasing  function $\F:\xR^+\times \xR^+\to\xR^+$ depending only on $d,~r,~s,~h$ such that
\[
P_T\le \F\lp M^1_{s,T}, M^2_{s,T}, Z^1_{r,T}, Z^2_{r,T}\rp\lA (\delta\eta, \delta\psi)\arrowvert_{t=0}\rA_{H^{s-1}\times H^{s-\tmez}}.
\]
\end{prop}
With the above ingredients we can prove our main theorem about the Cauchy problem.
\begin{thm} \label{theo:Cauchy}
	Let $d\geq 1$ and two real numbers $r,~s$ satisfying
\[
2<r< s-\frac d2+\mu,\quad
 \mu=\begin{cases} 
\frac{3}{20}~\text{if}~d=1,\\
\frac{3}{10}~\text{if}~d\ge 2.
\end{cases}
\]
Let $(\eta_0, \psi_0)\in H^{s+\mez}\times H^s$ be such that $\dist (\eta_0, \Gamma)>h>0$. Then there exists a time $T>0$ such that the Cauchy problem for~\eqref{eq:Zak} has a unique solution
\[
(\eta,\psi)\in L^{\infty}\lp[0, T]; H^{s+\mez}\times H^s\rp\cap L^p\lp[0, T]; W^{r+\mez, \infty}\times W^{r, \infty}\rp
\]
where $p=4$ when $d=1$ and $p=2$ when $d\ge 2$. Moreover, we have 
\[
(\eta,\psi)\in C^0\lp[0, T]; H^{s'+\mez}\times H^{s'}\rp,\quad\forall s'<s
\]
and 
\[
\inf_{t\in [0, T]}\dist(\eta(t), \Gamma)>h/2.
\]
\end{thm}
\begin{rem}
In view of the formulas \eqref{BV}, the initial velocity field in the Cauchy theory \ref{theo:Cauchy} may fail to be Lipschitz up to the  free surface but it becomes Lipschitz at almost all later time. This result is parallel to the result in \cite{ABZ} for pure gravity water waves.
\end{rem}
The plan of the paper is as follows. First, we prove in Section~\ref{sec:reduc} some reduction of the problem, reducing it to a semiclassical equation. In Section~\ref{sec:param}, we construct a microlocal parametrix for this equation. Then in Section~\ref{mainsection}, we use it to prove the Strichartz estimates. The last part, Section~\ref{sec:Cauchy}, is devoted to the local existence of solutions.
Some needed results about paradiferential calculus are recalled in Appendix~\ref{appendix}.

\section*{Acknowledgment}
\hk The authors would like to sincerely thank T.Alazard, N.Burq and C.Zuily for many fruitful discussions, suggestions in the preparation of this work, as well as their helpful comments on the final manuscript. Quang Huy Nguyen was partially supported by the labex LMH through the grant no ANR-11-LABX-0056-LMH in the "Programme des Investissements d'Avenir".

\section{Reductions of the system} \label{sec:reduc}
\subsection{Paradifferential reduction}
First of all, we recall precisely the paradifferential reduction of the gravity-capillary system \eqref{eq:Zak} that we performed in \cite{NgPo}, which requires the following symbols:
\begin{itemize}
\item Symbols of the Dirichlet-Neumann operator
\begin{align*}
\lambda^{(1)}&=\sqrt{(1+|\nabla \eta|^2)|\xi|^2-\lp \nabla \eta\cdot\xi\rp^2},\\
\lambda^{(0)}&=\frac{1+|\nabla \eta|^2}{2\lambda^{(1)}}\left\{ \div\lp \alpha^{(1)}\nabla\eta\rp+i\partial_\xi\lambda^{(1)}\cdot\nabla\alpha^{(1)}\right\}
\end{align*}
with 
\[
\alpha^{(1)}=\frac{\lambda^{(1)}+i\nabla\eta\cdot\xi}{1+|\nabla\eta|^2}.
\]
\item Symbols of the mean-curvature operator:
\[
{\ell}^{(2)}:=(1+|\nabla \eta|^2)^{-\frac{1}{2}}\left(|\xi|^2-\frac{(\nabla \eta\cdot\xi)^2}{1+|\nabla\eta|^2}\right), \quad {\ell}^{(1)}:=-\frac{i}{2}(\partial_x\cdot\partial_{\xi})\ell^{(2)}.
\]
\item Symbols used for symmetrization
\begin{gather*} 
q=\left(1+(\nabla_x\eta)^2 \right)^{-\frac{1}{2}}, \quad p=\left(1+(\nabla_x\eta)^2 \right)^{-\frac{5}{4}}|\xi|^\mez+p^{(-\mez)} ,
\end{gather*} 
where $p^{(-\mez)}=F(\nabla_x\eta, \xi)\partial^\alpha_x\eta$,  with $|\alpha|=2$ and 
$F\in C^\infty(\xR\times \xR\setminus\{0\}; \xC)$ is homogeneous of order $-1/2$ in $\xi$.
\item Symbols in the symmetrized equation:
\bq\label{def:gamma}
\begin{aligned}
&\gamma:=\sqrt{l^{(2)}\lambda^{(1)}}=\lp\la\xi\ra^2-\frac{\lp\nabla\eta\cdot\xi\rp^2}{1+\la\nabla\eta\ra^2}\rp^\frac{3}{4}, \\
& \omega:=-
      \frac{i}{2}\lp\partial_\xi\cdot\partial_x\rp\sqrt{l^{(2)}\lambda^{(1)}}+\sqrt{\frac{l^{(2)}}{\lambda^{(1)}}}\frac{\RE\lambda^{(0)}}{2}.
\end{aligned}
\eq
\end{itemize}
\begin{thm} 
Let $s>\frac{3}{2}+\frac{d}{2},~2<r<s-\frac d2+\mez$ and $p\in [1,\infty]$. Suppose that $(\eta,\psi)$ a solution of~\eqref{eq:Zak}  satisfying condition $(H_t)$ for all times~$t\in I$ and 
\begin{equation}\label{regu:eta,psi}
	(\eta,\psi)\in L^\infty(I;H^{s+\mez}(\R^d)\times H^s(\R^d))\cap L^p(I;W^{r+\mez,\infty}(\R^d)\times W^{r,\infty}(\R^d)),
\end{equation}
The complex-valued unknown $u:=T_p\eta+iT_q(\psi-T_B\eta)$ then
	satisfies 
	\begin{equation}\label{ww:reduce}
		\partial_tu+ T_V\cdot\nabla  u+iT\gamma u=f-iT_\omega u,
	\end{equation}
	where  for a.e. ~$t\in I$, 
	\begin{equation}\label{est:RHS}
		\lA f(t)\rA_{H^s}\leq\F\lp\lA\eta(t)\rA_{H^{s+\mez}},\lA\psi(t)\rA_{H^s}\rp\lb1+\lA\eta(t)\rA_{W^{r+\mez,\infty}}+\lA\psi(t)\rA_{W^{r,\infty}}\rb.
	\end{equation}
\end{thm}
As mentioned in the introductory section, we shall from now on consider \eqref{ww:reduce} as an independent equation with coefficients $V,~\gamma,~\omega,~\omega_1$ at the following regularity level
\begin{equation}\label{reguHolder:eta,psi}
\begin{gathered}
	V\in L^\infty\lp I; W^{\mez,\infty}(\xR^d)^d\rp\cap L^p\lp I; W^{1,\infty}(\xR^d)^d\rp,\\
	\eta \in L^\infty\lp I;W^{2,\infty}(\R^d)\rp\cap L^p\lp I;W^{\frac 52,\infty}(\R^d)\rp,
\end{gathered}
\end{equation}
which is sufficient for the semi-norms appearing in Theorem \ref{main:theo}.

We give here some preliminary informations on the principal symbol~$\gamma$.
Define
\begin{equation*}
	\C'=\lB\xi\in\R^d:1/4\leq\la\xi\ra\leq4\rB.
\end{equation*}
\begin{lem}\label{Hess:gamma}
	\begin{enumerate}
		\item The symbol~$\gamma$ is in~$L^\infty(I;W^{1,\infty}S^\tmez)\cap L^p(I;C^\tmez S^\tmez)$ and for all~$\beta\in\N^d$, 
		there exists~$\F_{\beta}:\R^+\rightarrow\R^+$ such that for all~$t\in I$, $\xi\in\C'$,
		      \begin{gather}
			\lA D^\beta_\xi\gamma(t,\xi)\rA_{W^{1,\infty}}\leq\F_{\beta}\lp\lA\nabla\eta(t)\rA_{W^{1,\infty}}\rp,	\label{eq:gammac1}\\
		      	\lA D^\beta_\xi\gamma(t,\xi)\rA_{C^\tmez}\leq\F_{\beta}\lp\lA\nabla\eta(t)\rA_{W^{1,\infty}}\rp
			      \lb1+\lA\nabla\eta(t)\rA_{W^{\tmez,\infty}}\rb.	\label{eq:gammactmez}
		      \end{gather}
		\item There exists an absolute constant $C_d>0$ such that  with~$c_0=C_d(1+ \lA\nabla\eta \rA_{L^\infty(I\times \R^d)})$  we have for all~$t\in I$, $x\in\R^d$, and~$\xi\in\C'$,
		      \begin{equation}
		      	\la\det\Hess_\xi(\gamma)(t,x,\xi)\ra\geq c_0.	\label{eq:gammahess}
		      \end{equation}
	\end{enumerate}
\end{lem}
\begin{proof}
The proof of part 1. is straightforward using product rules and Sobolev embedding. For a proof of part 2., we refer to  Corollary $4.7$ in \cite{ABZ}.
\end{proof}

\subsection{Localization in frequency}
\hk To prove our estimates, we will follow standard procedure: decomposing the solution using Littlewood-Paley theory and 
using a parametrix and a TT* argument to derive Strichartz estimates for those dyadic pieces.
We will then bring the pieces back together to derive a Strichartz estimate for the original solution. 
Standard definitions and notations for the Littlewood-Paley decomposition are recalled in appendix~\ref{ap:para}.
For~$j\geq 0$, the dyadic piece~$\Delta_ju$ verifies the equation
\begin{equation}\label{eq:sym}
	\lp\partial_t+iT_\gamma +T_V\cdot\nabla\rp\Delta_ju=F_j,
\end{equation}
where
\begin{equation}
	F_j:=\Delta_jF -i\Delta_j (T_{\omega} u)+i\lb T_\gamma,\Delta_j\rb+\lb T_V,\Delta_j\rb\cdot\nabla u.
\end{equation}
Recall that $\Delta_j\Delta_k =0$ if $|k-j|\ge 2$. In the sequel, we shall always consider $\Delta_ju$ for $j$ large enough, in particular $j\ge 1$ so the spectrum of $\Delta_j u$ is always contained in the annulus 
\[
\mathcal{C}_j:=\left\{\xi\in \xR^d: 2^{j-1}<|\xi|\le 2^{j+1}\right\}.
\]
Thanks to the spectral localization of $\Delta_j u$ we can replace the paradiferential operators with pseudodifferential operators. Such a replacement for the transport term is harmless due to the following lemma. 
\begin{lem}[\protect{\cite[Lemma~4.9]{ABZ}}]
	\label{lem:equipara}
We have 
		      \[
			T_V\cdot\nabla\Delta_ju=S_{j-3}(V)\cdot\nabla\Delta_ju+R_ju,
		      \]
		      where~$R_ju$ has its spectrum contained in an annulus ~$\mathcal{C}(c_12^{j-1},c_22^{j+1})$ and satisfies the following estimate
\[
	\lA R_ju\rA_{H^s(\R^d)}\leq C\lA V\rA_{W^{1,\infty}(\R^d)}\lA u\rA_{H^s(\R^d)}
\]
		      \label{it:equipara}
where the constant $C>0$ is independent of $u, V, j$.
\end{lem}
The preceding lemma was proved in \cite{ABZ} thanks to the special form of the symbol $V(x)\xi$. Here, for the highest order term, let us prove the following more general fact for any paradifferential operator. Let $a\in \Gamma^m_r,~r>0$ and define
\bq
\forall j\in \xZ,\quad S_j(a)(x,\xi)=\psi(2^{-j}D_x)a(x,\xi)
\eq
 the spatial regularization of the symbol~$a$, where~$\psi$ is given in the Littlewood-Paley decomposition~\ref{Paley}. 
\begin{prop}
 For every $j\in \xN^*$, define
$$ T_a\Delta_ju=S_{j-3}(a)(x, D_x)\Delta_ju+R_j'u.$$
Then the spectrum of $R''_ju$ is contained in  an annulus $\mathcal{C}(c_12^{j-1},c_22^{j+1})$ and for every $\mu\in \xR$ we have the following norm estimate
$$\lA R_j'u\rA_{H^{\mu-m+r}(\R^d)}\le CM^m_r(a)
			      \lA u\rA_{H^\mu(\R^d)}
$$
where the constant $C>0$ is independent of $a, u, j$.
\begin{rem}
If $a$ is homogeneous in $\xi$ then $S_{j-3}a$ is still homogeneous in $\xi$. This remark is important in the next part when we multiply both side of our equation by $h^\tdm$ to derive a semi-classical equation.
\end{rem}
\end{prop}
\begin{proof} Since $\varrho=1$ on the support of $\varphi_j$ for any $j\ge 1$, we see that
$$ R_j'u=T_a\Delta_ju-S_{j-3}(a)(x, D_x)\varrho(D_x)\Delta_ju.$$
In the following proof, we shall use the presentation of Métivier \cite{MePise} on pseudodifferential and paradifferential operators. To be compatible with \cite{MePise} we also abuse in notations: by $\Gamma^m_r$ we denote the class of symbols $a$ satisfying \eqref{para:symbol} for any $\xi\in \xR^d$ and by $M_0^m$ the semi-norm \eqref{defi:semi-norm} where the suppremum is taken over $\xi\in\xR^d$.\\
 \hk 1. By definition \eqref{eq.para} of the paradifferential operator $T_a$ we have $
T_a v=\Op(\sigma_a\varrho)v
$
where $\Op(\sigma_a\varrho)$ denotes the classical pseudodifferential operator with symbol
\[
\sigma_a(x, \xi)\varrho(\xi)=\chi(D_x, \xi)a(x,\xi)\varrho(\xi).
\]
Hence $R'_ju=\Op(a_j)u$ with
\[
a_j(x, \xi)=\sigma_a(x, \xi)\varrho(\xi)\varphi_j(\xi)-S_{j-3}(a)(x, \xi)\varrho(\xi)\varphi_j(\xi).
\]
 Now, we write 
\[
a_j=\big(\sigma_a\varrho\varphi_j-a\varrho\varphi_j\big)+\big (a\varrho\varphi_j- S_{j-3}(a)\varrho\varphi_j\big)=a_j^1+a_j^2.
\]
Applying Proposition $5.8 (ii)$ in \cite{MePise} gives $a_j^1\in \Gamma^{m-r}_0$ and (remark that $(\varphi_j)_j$ is bounded in $\Gamma^0_r$)
\[
M_0^{m-r}(a_j^1)\le CM_r^m(a\varrho\varphi_j)\le CM^m_r(a\rho).
\]
On the other hand, if we denote  $b=a\varrho\varphi$ then $a_j^2(x, \xi)=b(x,\xi)-\psi_{j-3}(D_x, \xi)b(x, \xi)$. Taking into account the fact that $\supp \varphi_j\subset B(0, C2^j)$  we may estimate
\begin{align*}
|a_j^2(x, \xi)|&\le \sum_{k\ge j-2}|\Delta_jb(x, \xi)|\le \sum_{k\ge j-2}
2^{-kr}\lA b(\cdot, \xi)\rA_{W^{r, \infty}}\\
&\le C2^{-jr}\lA b(\cdot, \xi)\rA_{W^{r, \infty}}=C2^{-jr}|\varphi_j(\xi)|\lA a(\cdot, \xi)\varrho(\xi)\rA_{W^{r, \infty}}\\
&\le C(1+|\xi|)^{m-r}M^m_r(a\varrho),\quad \forall \xi\in \xR^d.
\end{align*}
By the same method for estimating $|\partial^\alpha_\xi a_j^2|$ we obtain that $a_j^2\in\Gamma^{m-r}_0$ and hence $a_j\in \Gamma^{m-r}_0$; moreover 
\[
 M^{m-r}_0(a_j)\le CM^{m}_r(a\varrho).
\]
\hk 2.  
Property \eqref{chi:prop} implies in particular that
\[
\mathfrak{F}_x(\sigma_a)(\eta, \xi)=0~ \text{for}~|\eta|\ge \eps_2(1+|\xi|),
\]
here we denote $\mathfrak{F}_x$ the Fourier transform with respect the the patial variable $x$.\\
On the other hand, by definition of the smoothing operator 
\[
\mathfrak{F}_x S_{j-3}(a)(x, \xi)\varrho(\xi)\varphi_j(\xi)=\psi(2^{-(j-3)}\eta)\mathfrak{F}_xa(\eta, \xi)\varrho(\xi)\varphi(2^{-j}\xi)
\]
which vanish if $|\eta|\ge \mez(1+|\xi|)$. Indeed, if either $|\xi|> 2^{j+1}$ or $|\xi|\le 2^{j-1}$ then $\varphi(2^{-j}\xi)=0$. Considering $2^{j-1}<|\xi|\le 2^{j+1}$ then $|\eta|\ge \mez(1+|\xi|)>2^{j-2}$ and thus $\psi(2^{-(j-3)}\eta)=0$. We have proved the existence of $0<\eps<1$ such that 
\bq\label{spectral:a_j}
\mathfrak{F}_x a_j(\eta, \xi)=0 ~\text{for}~|\eta|\ge \eps(1+|\xi|).
\eq
\hk 3. By the spectral property \eqref{spectral:a_j} one can use the Bernstein inequalities (see Corollary $4.1.7$, \cite{MePise}) to prove
  that $a_j$ is a pseudodifferential symbol in the class $S^{m-r}_{1,1}$. Then, applying Theorem $4.3.5$ in \cite{MePise} we conclude that 
\[
\lA R'_ju\rA_{H^{\mu-m+r}(\xR^d)}=\lA \Op(a_j)u\rA_{H^{\mu-m+r}(\xR^d)}\le CM_0^{m-r}(a_j)\lA u\rA_{H^\mu(\xR^d)}.
\]
Finally, the Fourier transform of $R'_ju$  reads
\[
\mathfrak{F}(R'_ju)(\xi)=\int_{\xR^d}\mathfrak{F}_x(a_j)(\xi-\eta, \eta)\hat{u}(\eta)\d \eta.
\]
Using the spectral localization property \eqref{spectral:a_j} and the fact that $\mathfrak{F}_x(a_j)(\xi-\eta, \eta)$ contains the factor $\varphi_j(\eta)$ we conclude that the spectrum of $R'_ju$ is contained in an annulus of size $~2^j$ as claimed.
\end{proof}
Now we can use the preceding results to rewrite the equation as
\begin{gather} \label{eq:locedp}
	\lp\partial_t+iS_{j-3}(\gamma)(x,D_x)+ S_{j-3}(V)\cdot\nabla \rp\Delta_ju=F_j',\\
F_j'=F_j+ R_j+iR_j'. \label{Fj'}
\end{gather}
\subsection{Regularization of symbols}
\hk Now, following the classical method for quasilinear equations pioneered by Bahouri and Chemin 
in~\cite{BaCh-IMRN} and~\cite{BaCh-AJM}, we further regularize the equation,
using a parameter~$\delta\in(0,1)$. By doing so, we aim to construct a parametrix with a regular enough phase to apply the stationary phase argument. This results in a slightly worse remainder term,
which will in turn result in slightly worse Strichartz estimates. Eventually, we optimize in~$\delta$.

Define for all $(t,x,\xi)\in I\times\R^d\times\R^d$ and~$ j\geq 0$,
\[
	S_{{(j-3)}\delta}(\gamma)(t,x,\xi)=\psi(2^{-{(j-3)}\delta}D_x)\gamma(t,x,\xi)
\]
and similarly $S_{{(j-3)}\delta}(V)(t,x)$.
Let~$\ph_1\in C^\infty(\R^d)$, with
\begin{equation*}
	\quad\supp\ph_1\subset\C'=\lB\xi:\frac{1}{4}\leq\la\xi\ra\leq4\rB,\quad\ph_1=1
		\text{ on }\C'':=\lB\xi:\frac{1}{3}\leq\la\xi\ra\leq3\rB,
\end{equation*}
so that it is~$1$ on the support of the Littlewood-Paley function~$\ph$. Then, equation~(\ref{eq:locedp}) can be rewritten as
\begin{equation} \label{eq:regedp}
	L_j\Delta u_j(t,x):=\lp\partial_t+ S_{{(j-3)}\delta}(V)\cdot\nabla+iS_{{(j-3)}\delta}(\gamma)(x, D_x)\ph_1(2^{-j}D_x)\rp\Delta_ju=F_{j\delta},
\end{equation}
 with
\bq\label{eq:remreg}	
	F_{j\delta}=i\lp S_{{(j-3)}\delta}\gamma(x,D_x)-S_{j-3}\gamma(x, D_x)\rp\Delta_ju+F'_j+\lp S_{{(j-3)}\delta}(V)-S_{j-3}(V)\rp\cdot\nabla\Delta_ju.
\eq
The function~$\ph_1$ has been inserted to keep into the operator the information about the localization of its solution~$\Delta_ju$.\\
Next, Lemma \ref{Hess:gamma} shows that the Hessian in $\xi$ of $\gamma$ is non-degenerate and since $S_{j\delta}(\gamma)$ is a small perturbation of $\gamma$ when $j$ large enough, we also have
\begin{prop}	\label{prop:reggammahess}
	There exists~$c_0>0$, $j_0\in\N$, such that
	\begin{equation*}
		\la\det\Hess_\xi\lp S_{j\delta}\lp\gamma\rp\rp(t,x,\xi)\ra\geq c_0,
	\end{equation*}
	for all~$t\in I$, $x\in\R^d$, $\xi\in\C'$, $j\geq j_0$.
\end{prop}

\subsection{Semi-classical formulation}
\hk We now want to prove Strichartz estimates for the homogeneous version of equation \eqref{eq:regedp}:
\bq\label{eq:homtx}
L_ju_j(t,x)=0.
\eq
To this end, we  recast the problem in the semi-classical formalism with ~$h=2^{-j}$. One need to write the pseudodifferential operators as functions of~$hD_x$. Since the highest  order operator is of order~$\frac{3}{2}$, we will multiply the equation by~$h^\tmez$. Let us consider the following model operator
\begin{equation*}
	\partial_t+i\la D_x\ra^\tmez,
\end{equation*}
which becomes
\begin{equation*}
	h^\tmez\partial_t+i\la hD_x\ra^\tmez.
\end{equation*}
To give it the canonical form of a semi-classical equation, we need to put the equation in the semi-classical time~$\sigma:=h^{-\mez}t$. 
In our model, the operator would become
\begin{equation*}
	h\partial_\sigma+i\la hD_x\ra^\tmez.
\end{equation*}
Here, we have 
\begin{equation*}
	h^\tmez S_{{(j-3)}\delta}(\gamma)(t,x,D_x)\ph_1(2^{-j}D_x)=S_{j\delta}(\gamma)(h^\mez\sigma,x,hD_x)\ph_1(hD_x)
\end{equation*}
because of the homogeneity of the original symbol~$\gamma$, which is conserved by its spatial regularization. Next, for the change of temporal variable $t=h^\mez\sigma$ we set
\begin{gather}
	w_h(\sigma,x):=u_j(h^\mez\sigma,x), \quad V_h(\sigma,x):=S_{{(j-3)}\delta}(V)(h^\mez\sigma,x),\\
	\Gamma_h(\sigma,x,\xi):=S_{{(j-3)}\delta}(\gamma)(h^\mez\sigma,x,\xi)\ph_1(\xi),\\
\mathcal{L}_h(\sigma, x):=h\partial_\sigma+i\Gamma_h(x, hD_x)+h^\mez V_h\cdot(h\nabla_{\!x}),
\end{gather}
so that
\begin{equation}\label{L:t,sigma}
	h^\tdm (L_ju_j)(h^\mez \sigma, x)=\mathcal{L}_hw_h(\sigma, x),
\end{equation}
and we want to establish Strichartz estimates for the semi-classical PDE
\bq\label{eq:semieq}
\mathcal{L}_hw_h(\sigma, x)=0.
\eq
\paragraph{Symbolic calculus.}
To express the regularity of the symbols involved, we define for~$k\in\N$ and~$J$ a time interval, the quantities
\bq\label{N_kgamma}
\begin{gathered}
	\mathcal{N}_k(\gamma)(J):=\sum_{\la\beta\ra\leq k}\sup_{t\in J}\sup_{\xi\in\C'}\lA D^\beta_\xi\gamma\rA_{W_x^{1,\infty}(\R^d)},\\
	\mathcal{N}_k(\omega)(J):=\sum_{\la\beta\ra\leq k}\sup_{t\in J}\sup_{\xi\in\C'}\lA D^\beta_\xi\omega\rA_{L_x^\infty(\R^d)}.
\end{gathered}
\eq
The regularity of $V$ is tracked under the norm
\begin{equation}\label{E(J)}
	E(J):=L^p(J;W^{1,\infty}(\R^d))^d.
\end{equation}
To simplify notations, let us set 
\[
\Xi_k(J)=\lA V\rA_{E(J)}+\mathcal{N}_{k}(\gamma)(J).
\]
Let us define now our symbol classes.
\begin{defn} \label{def:hsymb}
	Let~$m\in\R$, $\mu_0\in\R^+$, and~$a(\sigma,x,\xi,h)$ a smooth function defined on~$h^{-\mez}J\times\R^d\times \mathcal{C}'\times\lp0,h_0\rb$, 
	with~$h_0>0$ and smooth in the second and the third argument.  We say that~$a\in S^m_{\mu_0}(h^{-\mez}J)$ (resp. $a\in\dot{S}^m_{\mu_0}(h^{-\mez}J)$)  if there is a 
		      function~$\F_k:\R^+\rightarrow\R^+$ for every~$k\in\N$ (resp. $k\in \N^*$), such that for 
		      all~$(\sigma,x,\xi,h)\in h^{-\mez}J\times\R^d\times\mathcal{C}'\times\lp0,h_0\rb$,
		      \begin{equation}	\label{eq:defsemsym}
			      \la D^\alpha_xD^\beta_\xi a(\sigma,x,\xi,h)\ra\leq\F_k\lp \Xi_{k+1}(J)\rp h^{m-\la\alpha\ra\mu_0},\quad \la\alpha\ra+\la\beta\ra\le k.
		      \end{equation}
\end{defn}
We need a result on composition of such symbols, whose proof is indeed the same as that of Proposition $4.20$, \cite{ABZ}.
\begin{prop} \label{prop:symbcomp}
	If~$f$ is a symbol in~$S_{\delta}^m(h^{-\mez}J)$ (respectively~$\dot{S}^m_\delta(h^{-\mez}J)$), with~$m\in\R$, and we are given two symbols~$U\in\dot{S}^\delta_\delta(h^{-\mez}J)$ and~$V\in\dot{S}^0_\delta(h^{-\mez}J)$,
	then
	\begin{equation*}
		F(\sigma,y,\zeta):=f\lp\sigma,U(\sigma,y,\zeta),V(\sigma,y,\zeta)\rp
	\end{equation*}
	is in~$S^m_\delta(h^{-\mez}J)$ (respectively~$\dot{S}^m_\delta(h^{-\mez}J)$).
\end{prop}

	Since $ \eta\in L^\infty(I, H^{s+\mez}(\R^d))$ and $V\in  L^\infty(I, H^s(\R^d))$ with $s>\tdm+\frac{d}{2}$ we obtain by using Bernstein inequalities that
\begin{lem}
We have $\Gamma_h,~\nabla_{\!x}\Gamma_h,~V_h\in S^0_\delta(h^{-\mez}I)$.
\end{lem}
\begin{rem}
Recall that~$\C''=\lB\xi\in\R^d:\frac{1}{3}\leq\la\xi\ra\leq3\rB$.
In the semi-classical scale, 	Proposition~\ref{prop:reggammahess} translates as
	\begin{equation}	\label{eq:ellGam}
		\la\det\Hess_\xi\lp\Gamma_h\rp(\sigma,x,\xi)\ra\geq c_0,
	\end{equation}
	for~$(\sigma,x,\xi,h)\in h^{-\mez}I\times\R^d\times\C''\times(0,h_0]$, for~$h_0$ small enough, because~$\ph_1=1$ on~$\C''$.
\end{rem}

\subsection{Straightening the transport term}
\hk The semi-classical equation~(\ref{eq:semieq}) is not perfectly adapted to the construction of a parametrix, the reason being the term 
of order~$h^\mez$, 
which has to be taken into account while constructing the phase. An easy way around this problem is to remark that this is only a transport term,
and can be straightened by going to the associated lagrangian coordinates. Consider the solution~$X_h(\sigma;y)\in\R^d$ of the differential equation
\begin{equation} \label{eq:straightedo}
	\lB
	\begin{aligned}
		\dot{X}_h(\sigma;y)&=h^\mez V_h(\sigma,X_h(\sigma;y)),\\
		X_h(0;y)&=y,
	\end{aligned}
	\right.
\end{equation}
where~$y\in\R^d$. The vector field~$V_h$ is in~$L^\infty(h^{-\mez}I;H^\infty(\R^d))^d$, and
\begin{equation*}
	\la h^\mez V_h(\sigma,x)\ra\leq C\lA V\rA_{L^\infty(I\times\R^d)},\quad\forall(\sigma,x)\in h^{-\mez}I\times\R^d.
\end{equation*}
Then~(\ref{eq:straightedo}) has a unique solution on~$h^{-\mez}I$. Moreover, we have the following estimates on the flow (for the proof, see Proposition $4.10$  in \cite{ABZ}).
\begin{prop}	\label{prop:regstrai}
	At fixed~$\sigma\in h^{-\mez}I$, the map~$y\mapsto X_h(\sigma;y)$ is in~$C^\infty(\R^d;\R^d)$, and there exists functions~$\F$, 
	$\F_\alpha:\R^+\rightarrow\R^+$ such that
	\begin{equation*} \tag{i}
	 	\lA\frac{\partial X_h}{\partial y}(\sigma;\cdot)-I\rA_{L^\infty}\leq\F\lp\lA V\rA_{E(I)}\rp\la h^\mez\sigma\ra^\mez,
	 \end{equation*}
	 \begin{equation}	\tag{ii}
	 	\lA(D^\alpha_yX_h)(\sigma;\cdot)\rA_{L^\infty}\leq
		  \F_\alpha\lp\lA V\rA_{E(I)}\rp h^{-\delta(\la\alpha\ra-1)}\la h^\mez\sigma\ra^\mez,\quad\la\alpha\ra\geq2.
	 \end{equation}
\end{prop}
\begin{cor}
	If $T$ satisfies 
\bq
\label{chooseT}
T\F(\lA V\rA_{E(I)})\ll 1
\eq
then for any $\sigma\in h^{-\mez}I$, the map~$X_h(\sigma)$ is a diffeomorphism from~$\R^d$ to itself.
\end{cor}
\begin{proof}
	Proposition~\ref{prop:regstrai} shows that for~$T$ small enough as in \eqref{chooseT}, the matrix   $\frac{\partial X_h}{\partial y}(\sigma;y)$ is invertible. Also, we have 
	\begin{equation*}
		\la X_h(\sigma;y)-y\ra\leq h^\mez\int_0^{h^{-\mez}T}\lA V(h^\mez s)\rA_{L^\infty(\R^d)}\d s\le T^{1/p'}\lA V\rA_{L^p(\lb 0,T\rb, L^\infty(\R^d))},
	\end{equation*}
with $1/p'+1/p=1$. Thus, the map~$X_h(\sigma)$ is proper. This enables us to conclude using the Hadamard theorem.

\end{proof}
We will always assume in what follows that the chosen $T$ satisfies \eqref{chooseT}. The Strichatz estimates for the original solution can be recovered by summing the ones for the short time, the number of pieces depending only on 
the $L^p_tL^\infty_x$-norms of~$V$ appearing in the final constant.\\
\hk Now we have to compute how our semi-classical equation~(\ref{eq:semieq}) gets affected by this change of variables. 
The new unknown will be~$v_h(\sigma,y):=w_h(\sigma,X_h(\sigma;y))$.
The important quantity is~$A:=(\Gamma_h(hD_x)w_h)(\sigma,X_h(\sigma;y))$.
Taking~$\sigma,h,\delta$ as parameters, we have
\begin{equation*}
	A=(2\pi h)^{-d}\int\int e^{ih^{-1}\lp X(y)-x'\rp\cdot\eta}\Gamma\lp X(y),\eta\rp w(x')\d x'\d\eta.
\end{equation*}
Now we need to set 
\begin{align*}
	&H_h(\sigma;y,y'):=\int_0^1\frac{\partial X_h}{\partial y}\lp\sigma;\lambda y+(1-\lambda)y'\rp\d\lambda,\quad M_h(\sigma;y,y'):=\lp H_h^T(\sigma;y,y')\rp^{-1},\\
	&M^0_h(\sigma;y):=\lp\lp\frac{\partial X_h}{\partial y}\rp^T(\sigma;y)\rp^{-1},\quad J_h(\sigma;y,y'):=\la\det\lp\frac{\partial X_h}{\partial y}\rp(\sigma;y')\ra\la\det M_h(\sigma;y,y')\ra.
\end{align*}
Proposition~\ref{prop:regstrai} shows that~$M$ and~$M^0$ are well defined. Remark that~$M^0(y)=M(y,y)$ and that~$J(y,y)=1$.
We now change variables in the expression of~$A$, putting~$x':=X(y')$. We will then use~$X(y)-X(y')=H(y,y')(y-y')$ and set~$\eta:=M(y,y')\zeta$ to get
\begin{equation*}
	A=(2\pi h)^{-d}\int\int e^{ih^{-1}(y-y')\cdot\zeta}\Gamma\lp X(y),M(y,y')\zeta\rp J(y,y')v(y')\d y'\d\zeta.
\end{equation*}
We have proved that 
\begin{equation}	\label{eq:straightopprinc}
	A=(\Gamma_h(hD_x)w_h)(\sigma,X_h(\sigma;y))=P_hv_h(\sigma,y),
\end{equation}
where~$P_h$ is a semi-classical pseudodifferential operator of amplitude
\begin{equation}\label{widetildep}
	\widetilde{p}_h(\sigma,y,y',\zeta):=\Gamma_h\lp\sigma,X_h(\sigma;y),M_h(\sigma;y,y')\zeta\rp J_h(\sigma;y,y').
\end{equation}
We define the symbol
\begin{equation}
	p_h(\sigma,y,\zeta):=\widetilde{p}_h(\sigma,y,y,\zeta)=\Gamma_h\lp\sigma,X_h(\sigma;y),M^0_h(\sigma;y)\zeta\rp.
\end{equation}
We also set
\begin{equation}
	p_h(\sigma,y,\zeta):=\widetilde{p}_h(\sigma,y,y,\zeta)=\Gamma_h\lp\sigma,X_h(\sigma;y),M^0_h(\sigma;y)\zeta\rp.
\end{equation}
Let us write
$I_h:= [0,h^{\mez+\delta}]
$
and impose a constrain on $\delta$: 
\bq\label{delta:cd1}
0<\delta \le \mez
\eq
so that for all $\sigma \in h^{-\mez}I_h$ one has
\bq\label{delta:coro1}
\la h^\mez\sigma\ra^\mez\le h^\delta.
\eq
\begin{prop}	\label{prop:eststraitsymb}
	For every~$k\in\N$, there exists~$\F_k:\R^+\rightarrow\R^+$, such that
	\begin{equation}	\label{eq:eststraitsymb}
		\la D^\alpha_yD^\beta_\zeta p_h(\sigma,y,\zeta)\ra\leq\F_k\lp\mathcal{N}_{k}(\gamma)(I)+\lA V\rA_{E(I)}\rp\lb1+h^{-(\la\alpha\ra-1)\delta}\rb,
	\end{equation}
	where~$\la\alpha\ra+\la\beta\ra\leq k$, and~$(\sigma,y,\zeta,h)\in h^{-\mez}I_h\times\R^d\times\mathcal{C}'\times\lp0,h_0\rb$.\\
Consequently, we have \begin{equation}\label{p-dp}
		p_h\in S^0_\delta(h^{-\mez}I_h)\quad\text{and }~\nabla_{\!y}p_h\in S^0_\delta(h^{-\mez}I_h).
	\end{equation} 
\end{prop}
\begin{rem}
Remark that Proposition~\ref{prop:symbcomp} implies only the first assertion in \eqref{p-dp}. In the construction of the phase of our parametrix below, to control the flow (see \eqref{est:M1}, \eqref{est:M2}) we need to differentiate $p$ twice in $x$ and thus the first assertion in \eqref{p-dp} implies only $\partial^2_xp\in S_{\delta}^{-2\delta}(h^{_\mez}I_h)$ while with the second one, we have $\partial^2_xp\in S_{\delta}^{-\delta}(h^{_\mez}I_h)$. Consequently, the restriction $\sigma\le h^\delta$ is sufficient instead of requiring $\sigma\le h^{2\delta} $. This means that the parametrix is constructed in a time of double length by virtue of the second one.
\end{rem}
\begin{proof}
	We will consider~$\sigma\in h^{-\mez}I_h$ and~$h\in (0, h_0)$ as parameters. Denote $A_k=\mathcal{N}_{k}(\gamma)(I)+\lA V\rA_{E(I)}$. First, remark that we can use the identity
	\begin{equation*}
		\lp\frac{\partial X}{\partial y}\rp(y)\cdot M^0(y)=I_d
	\end{equation*}
	and Proposition~\ref{prop:regstrai} to get 
	\begin{equation}	\label{eq:regM}
	\begin{aligned}
		&\lA M^0(y)-I\rA_{L^\infty(\R^d)}\leq\F\lp\lA V\rA_{E}\rp  h^\delta,\\
		&\lA D_y^\alpha M^0(y)\rA_{L^\infty(\R^d)}\leq\F_\alpha\lp\lA V\rA_{E}\rp h^{-\delta\la\alpha\ra}\la h^\mez\sigma\ra^\mez \leq\F_\alpha\lp\lA V\rA_{E}\rp h^{\delta-\delta\la\alpha\ra},\quad\text{for }\la\alpha\ra\geq1.
	\end{aligned}
	\end{equation}
	Now~$D^\beta_\zeta p$ is a finite linear combination of terms of the form
	\begin{equation*}
		(D^{\beta'}_\xi\Gamma)\lp X(y),M^0(y)\zeta\rp \cdot P_{\la\beta\ra}(M^0(y)):=A\cdot B,
	\end{equation*}
	where~$\la\beta'\ra=\la\beta\ra$, and where~$P_{\la\beta\ra}(M^0(y))$ is a homogeneous polynomial of order~$\la\beta\ra$ in the coefficients of~$M^0(y)$.
	Hence~$D^\alpha_yD^\beta_\zeta p$ is a finite linear combination of terms of the form
\[
D^{\alpha_1}_yA\cdot D^{\alpha_2}_yB,\quad \alpha_1+\alpha_2=\alpha.
\]
	Concerning $B$ we use~(\ref{eq:regM}) to find
	\begin{equation}	\label{eq:regB}
		\lB
		\begin{aligned}
			\la B\ra&\leq\F_k\lp\lA V\rA_{E}\rp,\\
			\la D_y^{\alpha_2}B\ra&\leq\F_k\lp\lA V\rA_{E}\rp h^{\delta-|\alpha_2|\delta} \quad\text{if }\la\alpha_2\ra\geq1.
		\end{aligned}
		\right.
	\end{equation}
	By the fact that $\Gamma_h\in S^0_\delta(h^{-\mez}I)$ we see that if $\alpha=0$ then $|AB|\le \F_k(\lA V\rA_{E})$. Considering now $|\alpha|\ge 1$. If $\alpha_1=0$ then $\alpha_2=\alpha\ne 0$ and thus by \eqref{eq:regB} 
\[
|A\cdot \partial_y^\alpha B|\le \F_k(\lA V\rA_{E})h^{\delta-\la\alpha\ra\delta}.
\]
From now on, we assume $\la\alpha_1\ra\geq1$. By virtue of the Faà di Bruno formula, we see that ~ $D^{\alpha_1}_yA$ is a finite linear combination of terms of the form
	\begin{equation*}
		C=\lp D^a_xD_\zeta^{\beta'+b}\Gamma\rp\lp X(y),M^0(y)\zeta\rp\prod_{j=1}^r\lp D_y^{l_j}X(y)\rp^{p_j}\lp D^{l_j}_yM^0(y)\zeta\rp^{q_j},
	\end{equation*}
	where~$1\leq\la a\ra+\la b\ra\leq\la\alpha_1\ra$, $\la l_j\ra\geq1$, $\sum_{j=1}^r(\la p_j\ra+\la q_j\ra)l_j=\alpha_1$, $\sum_{j=1}^rp_j=a$, $\sum_{j=1}^rq_j=b$.\\
	We distinguish 2 cases corresponding to $a=0$ or $a\ne 0$.

{\it Case 1}: $\la a\ra=0$. Then every~$p_j$ is~$0$, and~$\sum_{j=1}^r\la q_j\ra\la l_j\ra=\la\alpha_1\ra\geq1$, so that at least one of the~$\la q_j\ra\la l_j\ra$ is non null.
	Then using the boundedness of~$D_\zeta^{\beta'+b}\Gamma$ (since $\Gamma_h\in S^0_\delta(h^{-\mez}I)$),  estimates~(\ref{eq:regM}), and the fact that~$\zeta$ is bounded on the support of~$\ph_1(M^0(y)\zeta)$ and thus of~$(D_\zeta^{\beta'+b})\Gamma(X(y), M^0(y)\zeta)$, we obtain
	\begin{equation*}
		|C|\leq\F_k(A_{k}) h^{\delta-\delta\sum_{j=1}^r\la l_j\ra|q_j|} =\F_k (A_{k})h^{\delta-\delta|\alpha_1|}.
	\end{equation*}
On the other hand, \eqref{eq:regB} implies that $
|\partial_y^{\alpha_2} B|\le \F_k(\lA V\rA_{E})h^{-|\alpha_2|\delta},~ \forall \alpha_2\in \xN.
$ Therefore, we conclude in this case that 
\[
|C\cdot\partial_y^{\alpha_2} B|\le \F_k(A_{k})h^{\delta-|\alpha|\delta}.
\]
	{\it Case 2}: $\la a\ra\geq1$.  We use in this case ~$\nabla_{\!x}\Gamma\in S^0_\delta$, estimate (\ref{eq:regM}), and Proposition~\ref{prop:regstrai} with the remark above on the boundedness of~$\zeta$ to obtain
	\begin{align*}
		&\la\lp D^a_xD_\zeta^{\beta'+b}\Gamma\rp\lp X(y),M^0(y)\zeta\rp\ra
\leq\F_k (A_{k}) h^{-(|a|-1)\delta},\\
& \la \lp D_y^{l_j}X(y)\rp^{p_j}\ra \le \F_k(A_{k}) h^{-(|l_j|-1)|p_j|\delta},\\
&\la \lp D^{l_j}_yM^0(y)\zeta\rp^{q_j}\ra \le \F_k(A_{k}) h^{-|l_j||q_j|\delta}.
	\end{align*}
These estimates imply $|C|\le \F_k(A_{k})h^{M}$ with 
\[
M=-(|a|-1)\delta-\sum_{j=1}^r (|l_j|-1)|p_j|\delta-\sum_{j=1}^r |l_j||q_j|\delta=-(|\alpha_1|-1)\delta
\]
and hence $|C|\le \F_k(A_{k})h^{-(|\alpha_1|-1)\delta}$. The second inequality in \eqref{eq:regB} then yields
\[
|C\cdot\partial_y^{\alpha_2} B|\le \F_k(A_{k})h^{-(|\alpha|-1)\delta}.
\]
Summing up, we obtain in any case the desired estimate and complete the proof.
\end{proof}
We also have the following result, whose proof follows that of the preceding and is in fact simpler.
\begin{prop}	\label{prop:eststraitamp}
	For every~$k\in\N$, there exists~$\F_k:\R^+\rightarrow\R^+$, such that
\begin{equation} \label{eq:eststraitsymb2}
	\la D^{\alpha_1}_yD^{\alpha_2}_{y'}D^\beta_\zeta \widetilde {p_h}(\sigma,y, y',\zeta)\ra\leq\F_k\lp\mathcal{N}_{k}(\gamma)(I)+\lA V\rA_{E(I)}\rp h^{-(\la\alpha_1\ra +\la \alpha_2\ra)\delta},
\end{equation}

	where~$\la\alpha_1\ra+\la \alpha_2\ra+\la\beta\ra\leq k$, and~$(\sigma,y,\zeta,h)\in h^{-\mez}I_h\times\R^d\times\mathcal{C}'\times\lp0,h_0\rb$.
\end{prop}
Concerning the Hessian of the principal symbol, we derive the following result.
\begin{prop}	\label{prop:Hessp}
	There exist~$h_0>0$ and~$c_0>0$ such that
	\begin{equation*}
		\la\det\Hess_\zeta\lp p_h\rp\lp\sigma,y,\zeta\rp\ra\geq c_0,
	\end{equation*}
	for~$(\sigma,y,\zeta,h)\in h^{-\mez}I\times\R^d\times\C\times(0,h_0]$. Here, recall that~$\C=\lB\zeta\in\R^d:\mez\leq\la\zeta\ra\leq2\rB$.

\end{prop}
\begin{proof}
	The Hessians of~$p_h$ and~$\Gamma_h$ are conjugated by
	\begin{equation*}
		\Hess_\zeta(p_h)(y,\zeta)=\lp M^0(y)\rp^T\Hess_\zeta \Gamma_h(X(y),M^0(y)\zeta)M^0(y),
	\end{equation*}
	so the result follows from~(\ref{eq:regM}) and~(\ref{eq:ellGam}) for~$h_0$ small enough.
\end{proof}
At last, the transport term disappears, since
\begin{equation}	\label{eq:straightopconst}
	\lp h\partial_{\sigma}w_h+h^\mez V_h\cdot(h\nabla_{\!x}) w_h\rp\lp\sigma,X_h(\sigma;y)\rp=
	      h\partial_\sigma v_h(\sigma,y).
\end{equation}

Now, using~(\ref{eq:straightopprinc}) and~(\ref{eq:straightopconst}), the semi-classical equation~(\ref{eq:semieq}) becomes
\begin{equation}	\label{eq:semstraeq}
	\left(\mathcal{L}_hw_h\right)(\sigma, X_h(\sigma, y))=\lp h\partial_\sigma+iP_h\rp v_h(\sigma,y)=0
\end{equation}
via the change of spatial variable $v_h(\sigma,y):=w_h(\sigma,X_h(\sigma;y))$.

\section{Construction of the parametrix}	\label{sec:param}
We want to construct a parametrix for the operator~$h\partial_\sigma+iP_h$ (recall that the space-time variables are $(\sigma, y)$). To compensate for the loss in powers of~$h$ incurred while differentiating our symbols, we will need to restrict ourselves to a small time interval depending on the frequency and the number of derivative used to regularize: $\sigma\in h^{-\mez}I_h= [0, h^\delta]$.\\
 We will look for a parametrix with the following Fourier integral operator form
\begin{equation}	\label{eq:parame}
	\mathcal{K}v(\sigma,y)=\lp2\pi h\rp^{-d}\iint e^{ih^{-1}\lp\phi_h(\sigma,y,\eta)-y'\cdot\eta\rp}\widetilde{b_h}\lp\sigma,y,y',\eta\rp\chi_1(\eta) v(y')\d y'\d\eta.
\end{equation}
We will take~$\phi_h$ to be a real valued phase such that 
\begin{equation*}
	\phi_h(0,y,\eta)=y\cdot\eta
\end{equation*}
 and~$\widetilde{b_h}$ an amplitude of the form
\begin{equation}	\label{eq:amppara}
	\widetilde{b_h}(\sigma,y,y',\eta)=b_h(\sigma,y,\eta)\psi\lp\frac{\partial\phi_h}{\partial\eta}(\sigma,y,\eta)-y'\rp,
\end{equation}
where~$b|_{\sigma=0}=\chi(\eta)$, $\chi\in C^\infty_c(\R^d\setminus\lB0\rB)$ and~$\psi\in C^\infty_c(\R^d)$ is such that~$\psi(z)=1$ if~$\la z\ra\leq1$.
At last, $\chi_1\in C^\infty_c(\R^d)$ is~$1$ on the support of~$\chi$.

\subsection{Construction of the phase}\label{construction:phase}
As usual, the phase will be the solution of the eikonal equation associated with the principal symbol of the operator,
\begin{equation}	\label{eq:eik}
	\frac{\partial\phi_h}{\partial\sigma}+p_h\lp\sigma,y,\frac{\partial\phi_h}{\partial y}\rp=0,\quad\phi_h(0,y,\eta)=y\cdot\eta.
\end{equation}
We will solve this equation with the method of characteristics.
Those are the solution of the system
\begin{equation}	\label{eq:chareik}
	\lB
	\begin{aligned}
		\dot{y_h}(\sigma;y_0,\eta)&=\frac{\partial p_h}{\partial\zeta}\lp\sigma,y_h(\sigma;y_0,\eta),\zeta_h(\sigma;y_0,\eta)\rp,\quad&y_h(0;y_0,\eta)=y_0,\\
		\dot{\zeta_h}(\sigma;y_0,\eta)&=-\frac{\partial p_h}{\partial y}\lp\sigma,y_h(\sigma;y_0,\eta),\zeta_h(\sigma;y_0,\eta)\rp,\quad&
																						    \zeta_h(0;y_0,\eta)=\eta.
	\end{aligned}
	\right.
\end{equation}
This system has a unique solution on~$h^{-\mez}I_h$. Now let us show that for fixed~$h$, $\eta$, and~$\sigma$  this flow is a global diffeomorphism from~$\R^d$ to itself. \\
We start by showing that the differential of this map is invertible. Taking~$h$ as a parameter, denote by
\[
m(\sigma):=(\sigma,y(\sigma;y_0,\eta),\zeta(\sigma,y_0,\eta))
\]
 the flow-out of~$(0,y_0,\eta)$. Differentiate~(\ref{eq:chareik}) with respect to~$y_0$. Then at the point~$(y_0,\eta)$, there holds
\begin{equation} \label{eq:difchareik}
	\lB
	\begin{aligned}
		\dot{\frac{\partial y}{\partial y_0}}(\sigma)&=\frac{\partial^2p}{\partial y\partial\zeta}(m(\sigma))\frac{\partial y}{\partial y_0}(\sigma)
					    +\frac{\partial^2p}{\partial\zeta^2}(m(\sigma))\frac{\partial\zeta}{\partial y_0}(\sigma),\quad&\frac{\partial y}{\partial y_0}(0)=I_d,\\
		\dot{\frac{\partial\zeta}{\partial y_0}}(\sigma)&=-\frac{\partial^2p}{\partial y^2}(m(\sigma))\frac{\partial y}{\partial y_0}(\sigma)
					    -\frac{\partial^2p}{\partial\zeta\partial y}(m(\sigma))\frac{\partial\zeta}{\partial y_0}(\sigma),\quad&\frac{\partial\zeta}{\partial y_0}(0)=0.
	\end{aligned}
	\right.
\end{equation}
This system is linear, of the form~$\dot{U}(\sigma)=M(\sigma)U(\sigma)$. Then, Proposition~\ref{prop:eststraitsymb} and the remark that follows it give 
\begin{equation}\label{est:M1}
	\lA M(\sigma)\rA\leq\F\lp\mathcal{N}_{2}(\gamma)(h^{-\mez}I_h)+\lA V\rA_{E(h^{-\mez}I_h)}\rp h^{-\delta}.
\end{equation}
When we integrate in time over~$h^{-\mez}I_h=[0,h^\delta]$, we get
\begin{equation}\label{est:M2}
	\int_0^\sigma\lA M(s)\rA\d s\leq\F\lp\mathcal{N}_{2}(\gamma)(h^{-\mez}I_h)+\lA V\rA_{E(h^{-\mez}I_h)}\rp.
\end{equation}
The Gr\"{o}nwall inequality then shows that~$\lA U(\sigma)\rA$ is uniformly bounded on~$h^{-\mez}I_h$.
Now using~(\ref{eq:difchareik}) and noticing that the coefficients of the first equation involve only derivatives of order 0 and  1 in~$y$ of $p$,  we obtain by virtue of Proposition \ref{prop:eststraitsymb}
\begin{equation}	\label{eq:estderchar1}
		\la\frac{\partial y_h}{\partial y_0}(\sigma;y_0,\eta)-I_d\ra\leq\F\lp\mathcal{N}_{2}(\gamma)(h^{-\mez}I_h)+\lA V\rA_{E(h^{-\mez}I_h)}\rp h^\delta.
\end{equation}
	Similarly, since the second equation in \eqref{eq:difchareik} has coefficients containing derivatives of $p$ in $y$ up to order $2$, we have
\bq\label{eq:estderchar1'}
	\la\frac{\partial\zeta_h}{\partial y_0}(\sigma;y_0,\eta)\ra\leq\F\lp\mathcal{N}_{2}(\gamma)(h^{-\mez}I_h)+\lA V\rA_{E(h^{-\mez}I_h)}\rp.
\eq

Now taking~$h$ small enough, \eqref{eq:estderchar1} gives the invertibility of the matrix~$\frac{\partial y_h}{\partial y_0}(\sigma;y_0,\eta)$,
and since
\begin{equation*}
	\la y_h(\sigma;y_0,\eta)-y_0\ra\leq\int_0^\sigma\la\dot{y_h}(s,y_0,\eta)\ra\d s\leq\F\lp\mathcal{N}_{2}(\gamma)(h^{-\mez}I_h)+\lA V\rA_{E(h^{-\mez}I_h)}\rp h^\delta
\end{equation*}
for~$\sigma\in h^{-\mez}I_h$, the map $y_0\mapsto y_h(\sigma, y_0, \eta)$ is proper. Therefore it is as announced a global diffeomorphism, and we denote by~$\kappa_h$ its inverse:
\begin{equation*}
	y_h(\sigma;y_0,\eta)=y\Leftrightarrow y_0=\kappa_h(\sigma;y,\eta).
\end{equation*}
Then we can define for~$h_0$ small enough, for~$(h,\sigma,y,\eta)\in(0,h_0]\times\lp h^{-\mez}I_h\rp\times\R^d\times\R^d$ the real-valued phase
\begin{equation}	\label{eq:defphase}
	\phi_h(\sigma,y,\eta)=y\cdot\eta-\int_0^\sigma p_h\lp s,y,\zeta_h\lp s;\kappa_h(s;y,\eta),\eta\rp\rp\d s.
\end{equation}
\begin{prop}
	The function~$\phi$ defined in~(\ref{eq:defphase}) solves the eikonal equation~(\ref{eq:eik}).
\end{prop}
The proof of this proposition is standard (see for example~\cite{ZworskiSemClass}, 10.2.2).

The map~$\phi$ is~$C^1$ in~$\sigma$ and~$C^\infty$ in~$(y,\eta)$.
We can study the Hessian of this phase in the~$\eta$ variable, using our study of the symbol~$p$.
\begin{prop}	\label{prop:hess}
	For~$h_0$ small enough, there exists a constant~$M_0>0$ such that
	\begin{equation*}
		\la\det\Hess_\eta(\phi_h)(\sigma,y,\eta)\ra\geq M_0\sigma^d,
	\end{equation*}
	for~$\sigma\in\lp h^{-\mez}I_h\rp,$ $y\in\R^d$, $\eta\in\C$ and~$h\in(0,h_0]$.
\end{prop}
\begin{proof}
	By differentiating the eikonal equation~(\ref{eq:eik}) twice with respect to~$\eta$, we find that
	\begin{align*}
		\partial_\sigma(\partial_{\eta_i}\partial_{\eta_j}\phi)=&-\sum_{k,l=1}^d(\partial_{\zeta_k}\partial_{\zeta_l}p)\lp\sigma,y,\frac{\partial\phi}{\partial y}\rp(\partial_{y_k}\partial_{\eta_i}\phi)
			(\partial_{y_l}\partial_{\eta_j}\phi)\\&-\sum_{k=1}^d(\partial_{\zeta_k}p)\lp\sigma,y,\frac{\partial\phi}{\partial y}\rp(\partial_{y_k}\partial_{\eta_i}\partial_{\eta_j}\phi).
	\end{align*}
	From the initial conditions of the eikonal equation, we obtain the values of the terms at~$\sigma=0$, so that
	\begin{equation*}
		\partial_\sigma(\partial_{\eta_i}\partial_{\eta_j}\phi)|_{\sigma=0}=-(\partial_{\zeta_i}\partial_{\zeta_j}p)\lp0,y,\eta\rp,
	\end{equation*}
	so that
	\begin{equation*}
		\partial_{\eta_i}\partial_{\eta_j}\phi\lp\sigma,y,\eta\rp=-(\partial_{\zeta_i}\partial_{\zeta_j}p)\lp0,y,\eta\rp\sigma+\mathrm{o}(\sigma).
	\end{equation*}
	Then using Proposition~\ref{prop:Hessp} and taking~$h_0$ small enough, which means~$\sigma$ small, we can conclude the proposition.
\end{proof}
Now we want estimates of higher orders for the phase and various related quantities. We start by estimating the derivatives of the flow.
\begin{prop}\label{pro:estflow}
	There exists~$\F:\R^+\rightarrow\R^+$ such that
	\begin{equation}	\label{eq:estderchar2}
	\begin{aligned}
		\la\frac{\partial y_h}{\partial y_0}(\sigma;y_0,\eta)-I_d\ra&\leq\F\lp\mathcal{N}_{2}(\gamma)(h^{-\mez}I_h)+\lA V\rA_{E(h^{-\mez}I_h)}\rp h^\delta,\\
		\la\frac{\partial\zeta_h}{\partial y_0}(\sigma;y_0,\eta)\ra&\leq\F\lp\mathcal{N}_{2}(\gamma)(h^{-\mez}I_h)+\lA V\rA_{E(h^{-\mez}I_h)}\rp,\\
		\la\frac{\partial y_h}{\partial\eta}(\sigma;y_0,\eta)\ra&\leq\F\lp\mathcal{N}_{2}(\gamma)(h^{-\mez}I_h)+\lA V\rA_{E(h^{-\mez}I_h)}\rp h^\delta,\\
		\la\frac{\partial\zeta_h}{\partial\eta}(\sigma;y_0,\eta)-I_d\ra&\leq\F\lp\mathcal{N}_{2}(\gamma)(h^{-\mez}I_h)+\lA V\rA_{E(h^{-\mez}I_h)}\rp h^\delta,\\
	\end{aligned}
	\end{equation}
	for~$(h,s,y_0,\eta)\in(0,h_0]\times h^{-\mez}I_h\times\R^d\times\mathcal{C}'$.\\
	Moreover, for every~$k\geq1$ there exists~$\F_k:\R^+\rightarrow\R^+$ such that for~$(\alpha,\beta)\in\N^d\times\N^d$
 \[
  \left\{ 
  \begin{aligned}
   &  \la D_y^\alpha D_\eta^\beta  y_h(\sigma) \ra \leq \mathcal{F}_k\big(\Vert V \Vert_{E_0} +   \mathcal{N}_{k+1}(\gamma)\big) h ^{\delta-\delta|\alpha|},\\
  & \la D_y^\alpha D_\eta^\beta  \zeta_h(s)  \ra\leq  \mathcal{F}_k\big(\Vert V \Vert_{E_0} +   \mathcal{N}_{k+1}(\gamma)\big) h ^{-\delta|\alpha|}.
  \end{aligned}
  \right.
   \]
Consequently, 
	\begin{equation*}
		y_h\in\dot{S}^\delta_\delta(h^{-\mez}I_h),\quad\zeta_h\in\dot{S}^0_\delta(h^{-\mez}I_h).
	\end{equation*}
\end{prop}
\begin{proof}
	The first two estimates of~(\ref{eq:estderchar2}) have already been proven in~(\ref{eq:estderchar1}) and (\ref{eq:estderchar1'}) . Similarly, if we differentiate the characteristic system \eqref{eq:chareik} with respect to $\eta$ we obtain the last two estimates of~(\ref{eq:estderchar2}).
	
	To prove estimates on higher order derivatives, we proceed by induction and Gr\"{o}nwall inequality (see Proposition $4.21$, \cite{ABZ}).
\end{proof}
We now deduce estimates on some quantities associated to the phase.
Define
\begin{equation}	\label{eq:deftheta}
	\theta_h(\sigma,y,y',\eta):=\int_0^1\frac{\partial\phi_h}{\partial y}\lp\sigma,\lambda y+(1-\lambda)y',\eta\rp\d\lambda.
\end{equation}

\begin{cor}	\label{cor:estphi}
	For every~$k\geq1$ there exists~$\F_k:\R^+\rightarrow\R^+$ such that for~$(\alpha,\beta)\in\N^d\times\N^d$ with~$\la\alpha\ra+\la\beta\ra=k$,~$(\alpha_1,\alpha_2,\beta)\in\N^d\times\N^d\times\N^d$
	with~$\la\alpha_1\ra+\la\alpha_2\ra+\la\beta\ra=k$,
	\begin{align}
		\la D^\alpha_yD^\beta_\eta\kappa_h(\sigma;y,\eta)\ra&\leq\F_k\lp\mathcal{N}_{k+1}(\gamma)(h^{-\mez}I_h)+\lA V\rA_{E(h^{-\mez}I_h)}\rp h^{\delta-\la\alpha\ra\delta},	\label{eq:estkappa}\\
		\la D^\alpha_yD^\beta_\eta\lp\frac{\partial\phi_h}{\partial y}\rp(\sigma,y,\eta)\ra&\leq\F_k\lp\mathcal{N}_{k+1}(\gamma)(h^{-\mez}I_h)+\lA V\rA_{E(h^{-\mez}I_h)}\rp 
																					h^{-\la\alpha\ra\delta},	\label{eq:estderspph}\\
		\la D^\alpha_yD^\beta_\eta\lp\frac{\partial\phi_h}{\partial\eta}\rp(\sigma,y,\eta)\ra&\leq\F_k\lp\mathcal{N}_{k+1}(\gamma)(h^{-\mez}I_h)+\lA V\rA_{E(h^{-\mez}I_h)}\rp 
																					h^{\delta-\la\alpha\ra\delta}, \label{eq:estderimpph}\\
		\la D^{\alpha_1}_yD^{\alpha_2}_{y'}D^\beta_\eta\theta_h(\sigma,y,y', \eta)\ra&\leq\F_k\lp\mathcal{N}_{k+1}(\gamma)(h^{-\mez}I_h)+\lA V\rA_{E(h^{-\mez}I_h)}\rp 
																		h^{-\lp\la\alpha_1\ra+\la\alpha_2\ra\rp\delta},	\label{eq:esttheta}
	\end{align}
for all~$(h,\sigma,y,y',\eta)\in(0,h_0]\times\lp h^{-\mez}I_h\rp\times\R^d\times\R^d\times \mathcal{C}'$.
	
	This means that~$\kappa_h\in\dot{S}^\delta_\delta$, $\frac{\partial\phi_h}{\partial y}\in\dot{S}^0_\delta$, $\frac{\partial\phi_h}{\partial\eta}\in\dot{S}^\delta_\delta$.
\end{cor}
\begin{proof}
	The first estimate comes from the relation~$y(\sigma;\kappa(\sigma;y,\eta),\eta)=y$ which by differentiation gives
	\begin{equation*}
		\frac{\partial y}{\partial y}\cdot\frac{\partial\kappa}{\partial y}=I_d,\quad\frac{\partial y}{\partial y}\cdot\frac{\partial\kappa}{\partial\eta}=-\frac{\partial y}{\partial\eta}.
	\end{equation*}
	Now the case~$k=1$ follows from~(\ref{eq:estderchar2}), and by differentiating~$k+1$ times and using an induction we get~(\ref{eq:estkappa}).
	
	From the definition of~$\phi$ as a solution of the Hamilton-Jacobi equation associated to~$p$, we see that~$\phi$ is a generating function for the Lagrangian surface 
	\begin{equation*}
		\lB\lp\sigma,p(\sigma);y,\zeta(\sigma;\kappa(\sigma;y,\eta),\eta);\kappa(\sigma;y,\eta),\eta\rp|\sigma\in\lp h^{-\mez}I_h\rp,(y,\eta)\in\R^{2d}\rB,
	\end{equation*}
	so that
	\begin{equation}\label{dphi}
		\frac{\partial\phi}{\partial y}(\sigma,y,\eta)=\zeta(\sigma;\kappa(\sigma;y,\eta),\eta),\quad\frac{\partial\phi}{\partial\eta}(\sigma,y,\eta)=\kappa(\sigma;y,\eta).
	\end{equation}
	This immediately implies~(\ref{eq:estderimpph}), and since~$\zeta\in\dot{S}^0_\delta$, $\kappa\in\dot{S}^\delta_\delta$ and obviously~$\eta\in\dot{S}^0_\delta$, 
	we can use Proposition~\ref{prop:symbcomp} to get~(\ref{eq:estderspph}).
	
	At last~(\ref{eq:esttheta}) follows directly from the definition of~$\theta$ and from~(\ref{eq:estderspph}).
\end{proof}

\subsection{Construction of the amplitude}\label{construction:amplitude}
The first step in constructing the amplitude is to compute the expression

\begin{equation*}
	J_h(\sigma,y,y',\eta):=e^{-ih^{-1}\phi_h(\sigma,y,\eta)}P_h\lp e^{ih^{-1}\phi_h(\sigma,y,\eta)}\widetilde{b_h}\lp\sigma,y,y',\eta\rp\rp.
\end{equation*}
This is a classical computation, identical to the one performed in~\cite{ABZ}, section~4.7.1. 
This yields, taking~$h,\sigma,y',\eta$ as parameters, for all~$N\in\N^*$,
\begin{equation*}
\begin{aligned}
	J(y)=&p\lp y,\frac{\partial\phi}{\partial y}(y)\rp b(y)\Psi(y)+\sum_{1\leq\la\alpha\ra\leq N-1}\frac{h^{\la\alpha\ra}}{\alpha!}\left.D^\alpha_{z}\lb\lp\partial^\alpha_\eta\widetilde{p}\rp
																		      \lp y,z,\theta(y,z)\rp b(z)\rb\ra_{z=y}\Psi(y)\\&+U_N(y)+R_N(y)+S_N(y),
\end{aligned}     
\end{equation*}
where~$\Psi(y)=\psi\lp\frac{\partial\phi_h}{\partial\eta}(\sigma,y,\eta)-y'\rp$ which has been defined in~(\ref{eq:amppara}), and~$\theta$ has been defined in~(\ref{eq:deftheta}).
The first remainder contain all the terms where~$\Psi$ is differentiated at least once, 
\begin{equation}\label{eq:defUN}
	U_N(y):=\sum_{1\leq\la\alpha\ra\leq N-1}\frac{h^{\la\alpha\ra}}{\alpha!}\sum_{0\leq\la\beta\ra\leq\la\alpha\ra-1}\binom{\alpha}{\beta}
										  \left.D^\beta_{z}\lb\lp\partial^\alpha_\eta\widetilde{p}\rp\lp y,z,\theta(y,z)\rp b(z)\rb\ra_{z=y}D^{\alpha-\beta}_y\Psi(y),
\end{equation}
the second is the Taylor remainder due to the change of phase,
\begin{equation}	\label{eq:defRN}
	R_N(y):=\frac{1}{(2\pi h)^{d}}\iint e^{ih^{-1}(y-z)\cdot\mu}\kappa(\mu)r_N(y,z,\mu)\widetilde{b}(y)\d y\d\mu
\end{equation}
with
\[
r_N(y,z,\mu)=\sum_{\la\alpha\ra=N}\frac{N}{\alpha!}\int_0^1(1-\lambda)^{N-1}(\partial^\alpha_\eta\widetilde{p})(y,z,\mu+\lambda\theta(y,z))\mu^\alpha\d\lambda,
\] and where~$\kappa\in C^\infty_c$ is~$1$ on the
support of~$\widetilde{p}(y,z,\mu+\theta(z,z'))$, which is compact locally in~$\eta$ because the phase is locally bounded in~$\eta$.
The last one comes from this~$\kappa$ term, and it is
\begin{equation}\label{eq:defSN}
	S_N(y):=\frac{1}{(2\pi)^{d}}\sum_{\la\alpha\ra\leq N-1}\sum_{\la\beta\ra=N}N\frac{h^{\la\alpha\ra+\la\beta\ra}}{\alpha!\beta!}								\iint_0^1(1-\lambda)^{N-1}\mu^\beta\hat{\kappa}(\mu)f_{\alpha,\beta}(y,y+\lambda h\mu)\d\lambda\d\mu,
\end{equation}
where
\[f_{\alpha,\beta}(y,z)=\partial^\beta_zD^\alpha_{z}\lb\lp\partial^\alpha_\eta\widetilde{p}\rp\lp y,z,\theta(y,z)\rp\widetilde{b}(z)\rb.\]

Now~$\phi$ satisfies the eikonal equation~(\ref{eq:eik}), so that we have
\begin{multline}
	e^{-ih^{-1}\phi(y)}\lp h\partial_\sigma+iP\rp\lp e^{ih^{-1}\phi(y)}\widetilde{b}\lp y\rp\rp\\
	=h\lp\partial_\sigma b
	 +i\sum_{1\leq\la\alpha\ra\leq N-1}\frac{h^{\la\alpha\ra-1}}{\alpha!}\left.D^\alpha_{z}\lb\lp\partial^\alpha_\eta\widetilde{p}\rp\lp y,z,\theta(y,z)\rp b(z)\rb\ra_{z=y}\rp\Psi(y)\\
	 +\lp hb\partial_\sigma\Psi+iU_N\rp+iR_N+iS_N.
\end{multline}
We want this to be~$\mathcal{O}(h^{N+1})$. 
Let
\begin{equation}	\label{eq:TN1}
	T_N:=\partial_\sigma b
	 +i\sum_{1\leq\la\alpha\ra\leq N-1}\frac{h^{\la\alpha\ra-1}}{\alpha!}\left.D^\alpha_{z}\lb\lp\partial^\alpha_\eta\widetilde{p}\rp\lp y,z,\theta(y,z)\rp b(z)\rb\ra_{z=y},
\end{equation}
so that we want~$T_N=\mathcal{O}(h^N)$.

Writing
\begin{equation*}
	\mathcal{L}:=\partial_{\sigma}+\sum_{i=1}^da_i(y)\partial_{y_i}+c(y),
\end{equation*}
where
\begin{equation}\label{a0-c0}
	 \lB
	 \begin{aligned}
	 	a_i(y):=&\lp\partial_{\eta_i}p\rp\lp y,\frac{\partial\phi}{\partial y}(y)\rp,\\
	 	c(y):=&\sum_{i=1}^d\lp\partial_{\eta_i}\partial_{y_i'}\widetilde{p}\rp\lp y,y,\frac{\partial\phi}{\partial y}(y)\rp
			 +\sum_{i,j=1}^d\lp\partial_{\eta_i}\partial_{\eta_j}p\rp\lp y,\frac{\partial\phi}{\partial y}(y)\rp\lp\partial_{y_i}\partial_{y_j}\phi\rp(y),
	 \end{aligned}
	 \right.
\end{equation}
we can rewrite~(\ref{eq:TN1}) as
\begin{equation}	\label{eq:TN2}
	T_N=\mathcal{L}b(y)
	    +i\sum_{2\leq\la\alpha\ra\leq N-1}\frac{h^{\la\alpha\ra-1}}{\alpha!}\left.D^\alpha_{z}\lb\lp\partial^\alpha_\eta\widetilde{p}\rp\lp y,z,\theta(y,z)\rp b(z)\rb\ra_{z=y}.
\end{equation}
For a fixed $\nu$ satisfying
\bq\label{choose:nu}
0<\nu\le 1-\delta
\eq
we  look for~$b$ under the form
\begin{equation}\label{form:b}
	b=\sum_{k=0}^Nh^{k\nu}b_k.
\end{equation}
Inserting this ansatz into~(\ref{eq:TN2}) gives, after a change of indices,
\begin{equation}\label{TN:defn}
	\begin{aligned}
		T_N=&\sum_{k=0}^Nh^{k\nu}\mathcal{L}b_k(y)\\
		      &+i\sum_{k=1}^{N+1}h^{k\nu}\sum_{2\leq\la\alpha\ra\leq N-1}\frac{h^{\la\alpha\ra-1-\nu}}{\alpha!}\left.D^\alpha_{z}\lb\lp\partial^\alpha_\eta\widetilde{p}\rp\lp y,z,\theta(y,z)\rp b_{k-1}(z)\rb\ra_{z=y}.
	\end{aligned}
\end{equation}
We take~$b_0$ as a solution of
\begin{equation}	\label{eq:b_0}
	\lB
	\begin{aligned}
		\mathcal{L}b_0&=0,\\
		b_0(\eta)\rvert_{\sigma=0}&=\chi(\eta),
	\end{aligned}
	\right.
\end{equation}
where~$\chi\in C^\infty_c(\R^d\setminus\lB0\rB)$ is the Cauchy data needed for~$b$.

Then we will recursively construct~$b_k,~1\le k\le N$ as a solution of
\begin{equation}	\label{eq:b_k}
	\lB
	\begin{aligned}
		\mathcal{L}b_k&=F_{j-1}:=-i\sum_{2\leq\la\alpha\ra\leq N-1}\frac{h^{\la\alpha\ra-1-\nu}}{\alpha!}\left.D^\alpha_{z}\lb\lp\partial^\alpha_\eta\widetilde{p}\rp\lp y,z,\theta(y,z)\rp b_{k-1}(z)\rb\ra_{z=y}\\
	b_k\rvert_{\sigma=0}&=0.
	\end{aligned}
	\right.
\end{equation}
Again,  \eqref{eq:b_0} and \eqref{eq:b_k} are solved by the method of characteristics. First we study the highest-order coefficients.
\begin{lem}	\label{lem:ai}
	For~$1\leq i\leq d$, $a_i\in S^0_\delta(h^{-\mez}I_h)$.
\end{lem}
\begin{proof}
	We know from Proposition~\ref{prop:eststraitsymb} and the remark that follows it that~$\partial_{\eta_i}p\in S^0_\delta$. From Corollary~\ref{cor:estphi} 
	we have that~$\frac{\partial\phi}{\partial y}\in\dot{S}^0_\delta$. Then, using Proposition~\ref{prop:symbcomp} gives the lemma. 
\end{proof}

Now consider the differential equation
\begin{equation*}
	\lB
	\begin{aligned}
		\dot{Y}(\sigma)&=a(\sigma,Y(\sigma),\eta),\\
		Y(0)&=y.
	\end{aligned}
	\right.
\end{equation*}
From Lemma~\ref{lem:ai}, $a$ is bounded, so the system has a unique solution on~$h^{-\mez}I_h$.
Now we remark that
\begin{equation*}
	\frac{\partial a}{\partial y}=\frac{\partial^2p}{\partial y\partial\eta}+\frac{\partial^2p}{\partial\eta^2}\cdot\frac{\partial^2\phi}{\partial y^2}.
\end{equation*}
Thus, using Proposition~\ref{prop:eststraitsymb} and noticing as in the proof of Corollary~\ref{cor:estphi} that since~$\frac{\partial\phi}{\partial y}=\zeta(\kappa(y))$, we can differentiate to 
get~$\frac{\partial^2\phi}{\partial y^2}=\frac{\partial\zeta}{\partial y}\cdot\frac{\partial\kappa}{\partial y}$, which is bounded by (\ref{eq:estderchar2}) and \eqref{eq:estkappa}, we find 
that~$\frac{\partial a}{\partial y}$ is bounded.

Proceeding as in the proof of Proposition~\ref{prop:regstrai}, we differentiate the equation in $y$ and use Gr\"{o}nwall lemma to deduce that
\begin{equation*}
	\la\frac{\partial Y}{\partial y}-I_d\ra\leq\F\lp\mathcal{N}_{2}(\gamma)(h^{-\mez}I_h)+\lA V\rA_{E(h^{-\mez}I_h)}\rp h^\delta,
\end{equation*}
so that the map~$y\mapsto Y(\sigma;y,\eta)$ is a global diffeomorphism with inverse~$\mu(\sigma;Y,\eta)$.
Now again differentiating the equation and using the Faa-di-Bruno formula, we can prove by induction the following result.
\begin{lem}	\label{lem:Ymu}
	The functions~$Y$ and~$\mu$ both belong to~$\dot{S}^\delta_\delta (h^{-\mez}I_h)$.
\end{lem}

Now we see that
\begin{equation*}
	\frac{\d}{\d\sigma}\lb b_j(\sigma,Y(\sigma))\rb=\lp\frac{\partial b_j}{\partial\sigma}+a\cdot\nabla b_j\rp(\sigma,Y(\sigma))=-(cb_j)(\sigma,Y(\sigma))+F_{j-1}(\sigma,Y(\sigma)),
\end{equation*}
so that the unique solution of~(\ref{eq:b_0}) and~(\ref{eq:b_k}) is
\begin{equation}\label{formula:bj}
	\lB
	\begin{aligned}
		b_0(\sigma,y,\eta)&=\chi(\eta)\exp\lp\int_0^\sigma c\lp s,Y(s;\mu(\sigma,y,\eta),\eta), \eta\rp\d s\rp,\\
		b_j(\sigma,y,\eta)&=\int_0^\sigma e^{\int_\sigma^{s}c\lp s',Y(s';\mu(\sigma,y,\eta),\eta), \eta\rp\d s'}F_{j-1}\lp s,Y(s;\mu(\sigma,y,\eta),\eta), \eta\rp\d s.
	\end{aligned}
	\right.
\end{equation}
The main result in the construction of the amplitude is the following proposition on the regularity of the~$b_j$s.
\begin{prop}\label{prop:amppara}
	The symbols~$b_j$ are in~$S^0_\delta (h^{-\mez}I_h)$.
\end{prop}
\begin{proof}
{\it Step 1.}	We start by showing that 
	\begin{equation}	\label{eq:eintc}
		e^{\int_\sigma^{s}c\lp s',Y(s';\mu(\sigma,y,\eta),\eta), \eta\rp\d s'}\in S^0_\delta (h^{-\mez}I_h).
	\end{equation}
	This will be a consequence of the simple remark that if~$a\in S^m_\delta$ with~$m\geq0$, then~$e^a\in S^0_\delta$. So we need to show that 
	\begin{equation}	\label{eq:intc}
		\int_\sigma^{s}c\lp s',Y(s';\mu(\sigma,y,\eta),\eta),\eta\rp\d s'\in S^0_\delta (h^{-\mez}I_h).
	\end{equation}
	First, from Lemma~\ref{lem:Ymu} we now that~$Y$ and~$\mu$ both belong to~$\dot{S}^\delta_\delta$, and then Proposition~\ref{prop:symbcomp} implies that~$Y(s';\mu(\sigma,y,\eta),\eta)\in\dot{S}^\delta_\delta$. Therefore, once again, by Proposition~\ref{prop:symbcomp} we only need to prove 
\begin{equation}	\label{eq:intc'}
		\int_\sigma^{s}c( s', y,\eta)\d s'\in S^0_\delta (h^{-\mez}I_h).
	\end{equation}
(Recall \eqref{a0-c0} for the definition of $c$.)\\
	From Corollary~\ref{cor:estphi} we know that~$\frac{\partial\phi}{\partial y}\in\dot{S}^0_\delta$ and ~$\frac{\partial^2\phi}{\partial y^2}\in S^{-\delta}_\delta$. On the other hand, one can deduce from Proposition \ref{prop:eststraitsymb} that
\[
\partial_\eta^2p(y, \eta)\in S^0_\delta,\quad  (\partial_{y'}\partial_\eta \widetilde p)(y, y, \eta)\in S^0_\delta.
\]
 Then, by Proposition~\ref{prop:symbcomp} we get 
\[
(\partial^2_\eta p)(y, \partial_y\phi(y))\in S^0_\delta, \quad  (\partial_{y'}\partial_\eta \widetilde p)( y, y, \partial_y\phi(y))\in S^0_\delta
\]
and hence 
\[
(\partial^2_\eta p)(y, \partial_y\phi(y))\partial^2_y\phi(y)\in S^{-\delta}_\delta.
\]
Consequently, 
\[
\int_\sigma^s  (\partial^2_\eta p)(s', y, \partial_y\phi(s',y))\partial^2_y\phi (s', y)\d s'\in S^0_\delta, \quad  \int_\sigma^s(\partial_{y'}\partial_\eta \widetilde p)(s', y, \partial_y\phi(s', y), \eta)\d s'\in S^\delta_\delta.
\]
With this, we get~(\ref{eq:intc}) and thus~(\ref{eq:eintc}). \\
{\it Step 2.}	We now need to prove that
\bq\label{est:F-bk}
\int_0^\sigma F_{j-1}\lp s,Y(s;\mu(\sigma,y,\eta),\eta)\rp\d s\in S^0_\delta (h^{-\mez}I_h).
\eq
We write for any $2\le |\iota|\le N-1$,
\[
G_{j-1}(\sigma, y, \eta):=h^{\la \iota\ra-1-\nu}\left.D^\iota_{z}\lb\lp\partial^\iota_\eta\widetilde{p}\rp\lp \sigma, y,z,\theta(\sigma, y,z), \eta\rp b_{k-1}(z)\rb\ra_{z=y}.
\]

Since $Y\in \dot S^\delta_\delta$, by Proposition \ref{prop:symbcomp} we see that to obtain \eqref{est:F-bk} it suffices to prove that
 for any~$\Lambda=(\alpha, \beta)$ with~$\la \alpha\ra+\la \beta\ra=k\geq0$ there holds 
	\begin{align}\label{Gj}
		\la D^\Lambda G_{j-1}\lp s, y,\eta)\rp\ra\leq
		          \F_k&\lp \mathcal{N}_{k+1}(\gamma)(h^{-\mez}I_h)+\lA V\rA_{E(h^{-\mez}I_h)}\rp h^{-\delta-\la \alpha\ra\delta}
	\end{align}
where $D^\Lambda=(D^\alpha_y, D^\beta_\eta)$. Again, the Faa-di-Bruno formula implies that $D^\Lambda G_{j-1}$ is a finite linear combination of terms of the form $K_1\cdot K_2$ with
\begin{align*}
K_1&=h^{|\iota|-1-\nu} D^{\Lambda_1}\left\{\left.D^{\iota_1}_{z} \lb\lp\partial^\iota_\eta\widetilde{p}\rp\lp \sigma, y,z,\theta(y,z), \eta\rp \rb\ra_{z=y}\right\}, \\
K_2&=D^{\Lambda_2}D^{\iota_2}_zb_{j-1}(\sigma, y, \eta),
\end{align*}
where $\Lambda_i=(\alpha_i, \beta_i),~ |\Lambda_1|+|\Lambda_2|=|\Lambda|,~|\iota_1|+|\iota_2|=|\iota|$. By induction, there holds
\[
|K_2|\le \F_k(...)h^{-(|\alpha_2|+|\iota_2|)\delta}.
\]
On the other hand, thanks to Proposition~\ref{prop:eststraitamp}, we can deduce without any difficulty that 
\[
|K_1|\le \F_k(...)h^{|\iota|-1-\nu}h^{-(|\alpha_1|+|\iota_1|)\delta}.
\]
Consequently, $|K_1\cdot K_2|\le \F_k(...) h^M$ where, since $|\iota|\ge 2$,
\bq\label{expo:G1}
\begin{aligned}
M&=|\iota|-1-\nu-(|\alpha_1|+|\iota_1|)\delta-(|\alpha_2|+|\iota_2|)\delta=|\iota|(1-\delta)-1-\nu-|\alpha|\delta\\
& \ge 2(1-\delta)-1-\nu-|\alpha|\delta =1-\nu-2\delta-|\alpha|\delta\ge -\delta-|\alpha|\delta.
\end{aligned}
\eq
Therefore, we obtain \eqref{Gj}. 
\end{proof}
\begin{rem}
If instead of \eqref{form:b}, one takes $b$ of the usual form $b=\sum h^kb_k$ then a similar computation shows that step 2 of the above proof does not work.
\end{rem}
In summary, we have proved that
\begin{prop}
Let $\phi_h$ be the solution to the eikonal equation \eqref{eq:eik} and $b_j\in S^0_\delta (h^{-\mez}I_h)$ given by the formula~\eqref{formula:bj}. We have
\begin{equation}	\label{eq:remsparam}
	e^{-ih^{-1}\phi}\lp h\partial_\sigma+iP\rp\lp e^{ih^{-1}\phi}\widetilde{b}\rp=hT_N\Psi+ hb\partial_\sigma\Psi+iU_N+ iR_N+iS_N.
\end{equation}
with
\begin{equation}	\label{eq:TN}
	T_N=i h^{(N+1)\nu}\sum_{2\leq\la\alpha\ra\leq N-1}\frac{h^{\la\alpha\ra-1-\nu}}{\alpha!}\left.D^\alpha_{z}\lb\lp\partial^\alpha_\eta\widetilde{p}\rp\lp y,z,\theta(y,z)\rp b_N(z)\rb\ra_{z=y},
\end{equation}
and $U_N,~R_N,~S_N$ given by \eqref{eq:defUN}, \eqref{eq:defRN}, \eqref{eq:defSN} respectively.
\end{prop}
Define the "error" of the parametrix to the exact solution as
\begin{equation}\label{eq:Rh}
	R_h(\sigma,y):=\lp h\partial_\sigma+iP\rp\lp\mathcal{K}v\rp\lp\sigma,y\rp.
\end{equation}
Then using the preceding proposition and our study of the phase $\phi$ and the amplitude before, we can prove using the stationary phase method as in Proposition $4.31$, \cite{ABZ} that $\mathcal{K}v$ defined in \eqref{eq:parame} is a good parametrix in the following sense
\begin{prop}	\label{prop:param}
	Take~$M_0$ an integer. Then for any~$N\in \xN$, there exists a function~$\F_N:\R^+\rightarrow\R^+$ such that 
	\begin{equation*}
		\sup_{0<\sigma\leq h^\delta}\lA R_h(\sigma, \cdot)\rA_{H^{M_0}(\R^d)}\leq\F_N\lp\mathcal{N}_{k}(\gamma)+\lA V\rA_{E}\rp h^N\lA v\rA_{L^2}.
	\end{equation*}
\end{prop}
\begin{rem}
In  Proposition $4.31$, \cite{ABZ}, $\| v\|_{L^2}$ on the right hand side of the  preceding estimate is replaced by $\|v\|_{L^1}$. Let us remark how to get $\lA v\rA_{L^2}$ as above. According to Lemma $4.33$, \cite{ABZ} $R_h$ can be written as
\[
R_h(\sigma, y)=\int H_h(\sigma, y, y')v(y')dy'
\]
where the kernel $H_h$ satisfies the following property: let $ k_0>d$ be an integer then we have for some $\rho>0$
\bq\label{rem:R}
\sup_{\sigma\in[0, h^\delta], y\in \xR^d, y'\in \xR^d |\eta|\le C}\langle m(\sigma, y, y', \eta)\rangle^{k_0}\la D^\beta_y H_h(\sigma, y, y')\ra \le C_{\beta, N}h^{\rho N},
\eq
with
\[
m(\sigma, y, y', \eta)=\partial_\eta \phi(\sigma, y, \eta)-y'.
\]
We only need to bound $R_h$ in $L^2$, the bounds for $\partial^\beta_y R_h$ in $L^2$ follow similarly. Now by  the Schur test  it suffices to prove that 
\[
\sup_{\sigma \in [0, h^\delta], y'\in \xR}\int \la H_h(y, y')\ra dy\le C_Nh^{\rho N},\quad \sup_{\sigma \in [0, h^\delta], y\in \xR}\int \la H_h(y, y')\ra dy'\le C_Nh^{\rho N}
\]
In view of \eqref{rem:R} this reduces to 
\begin{align*}
&\sup_{\sigma \in [0, h^\delta], y'\in \xR, |\eta|\le C}\int \langle m(\sigma, y, y', \eta)\rangle^{-k_0}dy\le C_N h^{\rho N},\\
&\sup_{\sigma \in [0, h^\delta], y\in \xR, |\eta|\le C}\int \langle m(\sigma, y, y', \eta)\rangle^{-k_0}dy'\le C_N h^{\rho N}.
\end{align*}
The first inequality was proved in Lemma $4.32$, \cite{ABZ}. For the second one, the obvious change of variables $y'\mapsto \tilde y:=\partial_\xi \phi(\sigma, y, \eta)-y'$ gives the conclusion.
\end{rem}

\section{Strichartz estimates}\label{mainsection}
\hk We first derive Strichartz estimates for the semi-classical equation~(\ref{eq:semstraeq}).
If~$v_h^0$ is the initial datum for this equation, recall that the parametrix~$\mathcal{K}v_h^0$ is defined by~(\ref{eq:parame}), where  $\phi$ and $b$ were constructed in the preceding section. The kernel of $\mathcal{K}$ is
\begin{equation*}
	K_h(\sigma,y,y')=(2\pi h)^{-d}\int e^{ih^{-1}(\phi_h(\sigma,y,\eta))-y'\cdot\eta)}\widetilde{b_h}\lp\sigma,y,y',\eta\rp\chi_1(\eta)\d\eta,
\end{equation*}
so that
\begin{equation*}
	\mathcal{K}v_h^0=\int K_h(\sigma,y,y')v_h^0(y')\d y'.
\end{equation*}
The parametrix $\mathcal{K}$ at time 0 is a good approximation of the initial value, as proved below.
\begin{lem}
For any integer~$M_0$  greater than~$d/2$, we have
\begin{gather}	
	\mathcal{K}v_h^0(0,y)=v_h^0(y)+r_h(y),\label{eq:initparam}\\
	\lA r_h\rA_{H^{M_0}(\R^d)}\leq\F_N(\dots)h^N\lA v_h^0\rA_{L^2(\R^d)},\quad \forall N\in \N. \label{eq:reminitparam}
\end{gather}
\end{lem}
\begin{proof}
Keeping in mind the initial conditions imposed on~$\phi$ and~$b$, and the fact that~$v_h^0$ is localized in frequency, we have equation~(\ref{eq:initparam}) with
\begin{equation*}
	r_h(y)=\lp2\pi h\rp^{-d}\iint e^{ih^{-1}(y-y')\cdot\eta}\chi(\eta)\lp1-\Psi(y-y')\rp v_h^0(y')\d y'\d\eta.
\end{equation*}
Now for~$\la\beta\ra\leq M_0$, $D^\beta_yr_h(y)$ is a finite linear combination of terms of the form
\begin{equation*}
	h^{-d-\la\beta_1\ra}\iint e^{ih^{-1}(y-y')\cdot\eta}\eta^{\beta_1}\chi(\eta)\Psi_{\beta_1}(y-y')v_h^0(y')\d y'\d\eta,\quad\la\beta_1\ra\leq\la\beta\ra,
\end{equation*}
where~$\la y-y'\ra\geq1$ on the support of~$\Psi_{\beta_1}$.
This is a convolution of~$v_h^0(y')$ with 
\begin{equation*}
	w_h(Y):=h^{-d-\la\beta_1\ra}\int e^{ih^{-1}Y\cdot\eta}\eta^{\beta_1}\chi(\eta)\Psi_{\beta_1}(Y)\d\eta 
\end{equation*}
with~$\la Y\ra\geq1$ on the support of~$\Psi_{\beta_1}$.
This is an oscillating integral, and integrating by parts with the vector field
\begin{equation*}
	L=\frac{1}{\la Y\ra^2}\sum Y_j\partial_{\eta_j}
\end{equation*}
yields
\[
w_h(Y)=h^{M-d-\la\beta_1\ra}\int e^{ih^{-1}Y\cdot\eta}(-L)^M\lp \eta^{\beta_1}\chi(\eta)\rp\Psi_{\beta_1}(Y)\d\eta.
\]
Hence,  the~$L^1$ norm of $w_h$ is bounded by~$\F_M(\dots)h^{M-d-\la\beta_1\ra}$ for all~$M\in\N$, so that for~$\la\beta\ra\leq M_0$,
\begin{equation*}
	\lA D^\beta_yr_h\rA_{L^2}\leq\F_{M}(\dots)h^{M-d-\la\beta_1\ra}\lA v_h^0\rA_{L^2}.
\end{equation*}
This concludes the proof of~(\ref{eq:reminitparam}).
\end{proof}

Now define~$\mathcal{T}_h$ the propagator of our (homogeneous) semi-classical equation, i.e., 
\begin{equation}	\label{eq:defprop}
	\lB
	\begin{aligned}
		&\lp h\partial_\sigma+iP_h\rp\lp \mathcal{T}_h(\sigma,\sigma_0)v_h^0\rp(y)=0,\\
		&\lp  \mathcal{T}_h(\sigma_0,\sigma_0)v_h^0\rp(y)=v_h^0(y),
	\end{aligned}
	\right.
\end{equation}
where~$h\in(0,h_0]$ with~$h_0$ small enough, $0<\la\sigma-\sigma_0\ra\leq h^{\delta}$, $y\in\R^d$ and~$\widehat{v_h^0}$ supported in~$\C$.
Then using  the Duhamel formula and~\eqref{eq:Rh}, \eqref{eq:initparam} we can write 
\begin{equation}	\label{eq:duhpar}
	\mathcal{T}_h(\sigma,\sigma_0)v_h^0=\mathcal{K}v_h^0(\sigma-\sigma_0)-\mathcal{T}_h(\sigma,\sigma_0)r_h-\int_{\sigma_0}^\sigma \mathcal{T}_h(\sigma,s)R_h(s)\d s.
\end{equation}
By classical energy estimates, we have that~$\mathcal{T}_h$ is bounded on Sobolev spaces (and notably on~$L^2$), uniformly in time.
This, combined with Proposition~\ref{prop:param}, \eqref{eq:reminitparam} and~\eqref{eq:duhpar}, gives
\begin{equation*}
	\sup_{0\leq \sigma\leq h^{\delta}}\lA\mathcal{K}v_h^0(\sigma)\rA_{L^2}\leq\F(\dots)\lA v_h^0\rA_{L^2}.
\end{equation*}

Thus to use the classical TT* argument and prove Strichartz estimates, we only need to prove the following lemma.
\begin{lem}
	There holds for any~$0<\sigma\leq h^{\delta}$, 
	\begin{equation*}
		\lA\mathcal{K}(\sigma)\mathcal{K}^*(\sigma')v_h^0\rA_{L^\infty(\R^d)}\leq \F\lp\Xi_k\rp h^{-\frac{d}{2}}
		  \la\sigma-\sigma'\ra^{-\frac{d}{2}}\lA v_h^0\rA_{L^1(\R^d)},
	\end{equation*}
	where~$\mathcal{K}^*$ denotes the adjoint of~$\mathcal{K}$.
\end{lem}
\begin{proof}
	Here we follow the proof of Theorem $10.8$,~\cite{ZworskiSemClass}. The bound of the Lemma will be implied by an~$L^\infty$ bound on the kernel of~$\mathcal{K}(\sigma)\mathcal{K}^*(\sigma')$, which is
	\[W(\sigma,\sigma',x,z):=\frac1{(2\pi h)^{2d}}\iiint e^{\frac ih(\phi_h(\sigma,x,\eta)-\phi_h(\sigma',z,\zeta)-y\cdot(\eta-\zeta))}B\d y\d\zeta\d\eta\]
	where~$B\in S^0_\delta$. 
	
	In the~$(y,\zeta)$ variables, $\phi_h$ is non-degenerate and it is stationary at~$\zeta=\eta$, $y=\partial_\zeta\phi_h(\sigma',z,\zeta)$. Thus using stationary phase we get
	\[W(\sigma,\sigma',x,z)=\frac1{(2\pi h)^{d}}\int e^{\frac ih(\phi_h(\sigma,x,\eta)-\phi_h(\sigma',z,\eta))}B'\d y\d\zeta\d\eta\]
	where again~$B'\in S^0_\delta$.
	
	Now the phase of this oscillating integral is 
	\begin{align*}
		\tilde{\phi}:=&\phi_h(\sigma,x,\eta)-\phi_h(\sigma',z,\eta)\\
			     =&(\sigma-\sigma')(p_h(0,x,\eta)+\mathcal{O}(\la s\ra+\la s'\ra))+\ls x-z,\eta+\sigma'F(\sigma',x,z,\eta)\rs
	\end{align*}
	where $F$ is in~$S^0_\delta$ with seminorms controled by~$\F\lp\Xi_k\rp$, and the constant of the~$\mathcal{O}$ is also of this form.
	The phase is stationary when
	\[\partial_\eta\tilde{\phi}=(I+\sigma'\partial_\eta F)(x-z)+(t-s)(\partial_\eta p_h+\mathcal{O}(\la\sigma\ra+\la\sigma'\ra))=0,\]
	and since for~$h$ small, $\sigma'$ is small and thus~$(I+\sigma'\partial_\eta F)$ is invertible, the phase can only be stationary when~$x-z=\mathcal{O}(\sigma-\sigma')$.
	The Hessian is then
	\[\partial^2_\eta\tilde{\phi}=(\sigma-\sigma')(\partial^2_\eta p_h(0,x,\eta)+\mathcal{O}(\la\sigma\ra+\la\sigma'\ra)).\]
	Since~$\partial^2_\eta p_h(0,x,\eta)$ is non-singular, stationary phase gives for~$\la\sigma-\sigma'\ra>Ch$ that
	\[\la W(\sigma,\sigma',x,z)\ra\leq\F\lp\Xi_k\rp h^{-\frac d2}\la\sigma-\sigma'\ra^{-\frac d2},\]
	and for~$\la\sigma-\sigma'\ra<Ch$, it is easily seen that
	\[\la W(\sigma,\sigma',x,z)\ra\leq\F\lp\Xi_k\rp h^{-d}\leq\F\lp\Xi_k\rp h^{-\frac d2}\la\sigma-\sigma'\ra^{-\frac d2}\]
	also. This concludes the proof of the Lemma.
\end{proof}
Now the TT* argument (see Theorem $10.7$,~\cite{ZworskiSemClass}) can be invoked to prove the Strichartz estimates for the parametrix :
\begin{prop}
	For any~$2\leq p\leq\infty$, $1\leq q\leq\infty$ such that
	\[\frac 2p+\frac dq=\frac d2,\]
	there is a nonnegative nondecreasing function~$\F$ such that for~$0<h\leq h_0$ small enough, there holds
	\[\lA \mathcal{K}v_h^0\rA_{L^p((0,h^\delta);L^q(\R^d))}\leq\F(\Xi_k)h^{-\frac 1p}\lA v_h^0\rA_{L^2(\R^d)}.\]
\end{prop}
This, Proposition~\ref{prop:param}, \eqref{eq:reminitparam} and~\eqref{eq:duhpar}, the boundedness of~$\mathcal{T}_h$ on Sobolev spaces and Sobolev embeddings give the same Strichartz estimates for~\eqref{eq:defprop}.
\begin{cor}	\label{stri:v}
	For any~$2< p\leq\infty$, $1\leq q\leq\infty$ such that
	\[\frac 2p+\frac dq=\frac d2,\]
	there is a nonnegative nondecreasing function~$\F$ such that for~$0<h\leq h_0$ small enough, for any~$\sigma_0\in h^{-\mez}I$, there holds
	\[\lA \mathcal{T}_hv_h^0\rA_{L^p((\sigma_0,\sigma_0+h^\delta);L^q(\R^d))}\leq\F(\Xi)h^{-\frac 1p}\lA v_h^0\rA_{L^2(\R^d)}.\]
\end{cor}

Now, recall from \eqref{eq:semstraeq} that with $v_h(\sigma, y)=w_h(\sigma, X_h(\sigma, y))$ there holds
\[
	\left(\mathcal{L}_hw_h\right)(\sigma, X_h(\sigma, y))=\lp h\partial_\sigma+iP_h\rp v_h(\sigma,y)=0.
\]
Denoting by $\mathcal{S}_h(\sigma, \sigma_0)$ the flow map of $\mathcal{L}_hw_h(\sigma, x)=0$ we deduce immediately from Corollary \ref{stri:v} the following estimates.
\begin{cor}\label{stri:w}
	Let~$\chi\in C^\infty_c(\R^d)$ be supported in~$\C=\lB\xi:\mez\leq\la\xi\ra\leq2\rB$. 
	Take~$2< p\leq\infty$, $1\leq q\leq\infty$ such that
	\[\frac 2p+\frac dq=\frac d2.\]
	There is a nonnegative nondecreasing function~$\F$ such that for~$0<h\leq h_0$ small enough, for any~$\sigma_0\in h^{-\mez}I$, there holds with
	$w_h^0:=\chi(hD_y)w_0$, for any~$L^2$ function~$w_0$, that
	\[\lA \mathcal{S}_hw_h^0\rA_{L^p((\sigma_0,\sigma_0+h^\delta);L^q(\R^d))}\leq\F(\Xi_k)h^{-\frac 1p}\lA w_h^0\rA_{L^2(\R^d)}.\]
\end{cor}
We are now in position to derive Strichartz estimates for the operator $L_j$, whose flow map is denoted by $\mathcal{S}_j$, using the relation \eqref{L:t,sigma}:
\[
	h^\tdm (L_ju_j)(h^\mez \sigma, x)=\mathcal{L}_h(\sigma, x)w_h(\sigma, x), \quad w_h(\sigma, x)=u_j(h^\mez \sigma,x)~h=2^{-j}.
\]

\begin{thm}	\label{thm:Strshort}
	Let~$I_j=[t_0,t_0+2^{-j(\delta+\mez)}]$.
	There exist~$k\in\N$, and~$j_0\in\N$ such that for any~$s\in\R$ and~$\eps>0$ there exist~$\F,\F_{\eps}:\R^+\rightarrow\R^+$ such that if we have 
	\begin{equation*}
		\lB
		\begin{aligned}
			&L_ju_j=0,\\
			&u_j(t_0)=u_j^0,
		\end{aligned}
		\right.
	\end{equation*}
	where~$u_j$, $u_j^0$ and~$F_{j\delta}$ are supported in the annulus~$\C_j=\lB\xi:\frac{1}{C}{2^j}\leq\la\xi\ra\leq C2^j\rB$, then there exist $k=k(d)$ and $j_0\in \xN$ such that  for $j\ge j_0$, we have
	\begin{gather*}
		      	\lA u_j\rA_{L^4(I_j,W^{s-\frac{1}{8},\infty}(\R))}\leq\F(N)\lA u_j^0\rA_{H^s(\R)}, \quad\text{if}~d=1,\\
		      	\lA u_j\rA_{L^{2+\eps}(I_j,W^{s-\frac{d}{2}+\frac{3}{4}-\eps,\infty}(\R^d))}\leq\F_{\eps}(N)\lA u_j^0\rA_{H^s(\R^d)}, \quad\text{if}~d= 2.
	\end{gather*}
\end{thm}
We now glue together the Strichartz estimates in the preceding proposition to get Strichartz estimates in the full time interval.
\begin{cor}	\label{cor:Strfull}
	Recall that~$I=[0,T]$.  Put
$
\varsigma=\mez+\delta.
$
	There exist~$k\in\N$, and~$j_0\in\N$ such that for any~$s\in\R$ and~$\eps>0$, there exist~$\F,\F_{\eps}:\R^+\rightarrow\R^+$ such that if we have 
	\begin{equation}\label{eq:regu}
		\lB
		\begin{aligned}
			&L_ju_j(t,x)=F_{j\delta}(t,x),\\
			&u_j(t_0, x)=u_j^0(x),
		\end{aligned}
		\right.
	\end{equation}
	where~$u_j$, $u_j^0$ and~$F_{j\delta}$ are supported in the annulus~$\C_j=\lB\xi:\frac{1}{C} {2^j}\leq\la\xi\ra\leq C2^j\rB$, then there exist $k=k(d)$ and $j_0\in \xN$ such that  for $j\ge j_0$, there holds
	\begin{gather*}
		      	\lA u_j\rA_{L^4(I,W^{s-\frac{1}{2}+\frac{3}{8}-\frac{\varsigma}{4},\infty}(\R))}\leq\F((\Xi_k)\lp\lA F_{j\delta}\rA_{L^4(I,H^{s-\varsigma}(\R))}+\lA u_j\rA_{L^\infty(I,H^s(\R))}\rp,~\text{if}~d=1,\\
		      	\lA u_j\rA_{L^2(I,W^{s-\frac{d}{2}+\frac{3}{4}-\frac{\varsigma}{2}-\eps,\infty}(\R^d))}\leq\F_{\eps}(\Xi_k)\lp\lA F_{j\delta}\rA_{L^2(I,H^{s-\varsigma}(\R^d))}+\lA u_j\rA_{L^\infty(I,H^s(\R^d))}\rp,d\ge 2.
\end{gather*}
\end{cor}
\begin{proof}
	Take a cut-off~$\chi\in C^\infty_c(0,2)$ equal to one on~$[\mez,\tmez]$. Define for~$0\leq m\leq[2^{\varsigma j}T]-2$ the interval~$I_{j,m}:=[m2^{\varsigma j},(m+2)2^{-\varsigma j}]$, and 
	the associated cut-off~$\chi_{j,m}(t):=\chi\lp\frac{t-m2^{-\varsigma j}}{2^{-\varsigma j}}\rp$. We have
\[
L_j(\chi_{j,m}u_j)=\chi_{j,k}F_{j,\delta}+2^{\varsigma j}\chi'\lp\frac{tm2^{-\varsigma j}}{2^{-\varsigma j}}\rp u_j,
\]
	with~$\chi_{j,m}u(k2^{-\varsigma j})=0$.
	Then applying Theorem~\ref{thm:Strshort} to~$\chi_{j,k}u_j$ with the help of the Duhamel formula, noticing that the flow maps $\mathcal{S}(t, \tau)$ are bounded on Sobolev spaces and ~$\chi_{j,m}=1$ on~$((m+\mez)2^{\varsigma j},(m+\tmez)2^{-\varsigma j})$, we find for~$d\geq2$
	\begin{align*}
		&\lA u_j\rA_{L^2(((m+\mez)2^{-\varsigma j},(m+\tmez)2^{-\varsigma j}),W^{s-\frac{d}{2}+\frac{3}{4}-\eps,\infty})}\\
		  &\quad\leq\F(\Xi_k)\lp\lA F_{j\delta}\rA_{L^1((m2^{-\varsigma j},(m+2)2^{-\varsigma j}),H^s)}+2^{\varsigma j}\lA\chi'\lp\frac{t-m2^{-\varsigma j}}{2^{-\varsigma j}}\rp u_j\rA_{L^1(I,H^s)}\rp\\
		  &\quad\leq\F(\Xi_k)\lp2^{-\varsigma j/2}\lA F_{j\delta}\rA_{L^2((m2^{-\varsigma j},(m+2)2^{-\varsigma j}),H^s)}+\lA u_j\rA_{L^\infty(I,H^s)}\rp.
	\end{align*}
	Then we multiply both sides by~$2^{-\varsigma j/2}$ and use the fact that~$u_j$ and~$F_{j\delta}$ are supported in annulus to find 
	\begin{multline*}
		\lA u_j\rA_{L^2(((m+\mez)2^{-\varsigma j},(m+\tmez)2^{-\varsigma j}),W^{s-\frac{d}{2}+\frac{3}{4}-\frac{\varsigma}{2}-\eps,\infty})}\\
		\leq\F(\Xi_k)\lp\lA F_{j\delta}\rA_{L^2((m2^{-\varsigma j},(m+2)2^{-\varsigma j}),H^{s-\varsigma })}+2^{\varsigma j/2}\lA u_j\rA_{L^\infty(I,H^s)}\rp.
	\end{multline*}
	At last, elevating at the power~$2$ and summing back the pieces, and adding the control of the first and last pieces using Theorem~\ref{thm:Strshort}, we find the result as claimed.
	
	The case~$d=1$ follows along the same lines.
\end{proof}
The next step is to  derive Strichartz estimates for the non-regularized equation. For these estimates, one need  the following higher order semi-norm of $\gamma$
\begin{equation}\label{M_kgamma}
	\mathcal{M}_k(\gamma)(J):=\sum_{\la\beta\ra\leq k}\sup_{\xi\in\C'}\lA D^\beta_\xi\gamma\rA_{L^p(J;W_x^{\tdm,\infty}(\R^d))},
\end{equation}
and we put
\[
\widetilde\Xi =\mathcal{M}_k(\gamma)(J)+\mathcal{N}_k(\gamma)(J)+\mathcal{N}_k(\omega)(J)+\lA V\rA_{E(J)}.
\]
\begin{cor}	\label{cor:Strloc}
	There exists~$k\in\N$, and~$j_0\in\N$ such that for any~$s\in\R$ and~$\eps>0$ there exist~$\F,\F_{\eps}:\R^+\rightarrow\R^+$ such that if we have 
	\begin{equation}\label{eq:nonregu}
		\lB
		\begin{aligned}
			&\lp\partial_t+iT_\gamma+T_V\cdot\nabla\rp u_j=F_j,\\
			&u_j(0)=u_j^0,
		\end{aligned}
		\right.
	\end{equation}
	where~$u_j$, $u_j^0$ and~$F_j$ are supported in the annulus~$\C_j=\lB\xi:\frac{1}{C}{2^j}\leq\la\xi\ra\leq C2^j\rB$, then there exist $k=k(d)$ and $j_0\in \xN$ such that  for $j\ge j_0$, we have
	\begin{itemize}
		\item if~$d=1$,
		      \begin{equation*}
		      	\lA u_j\rA_{L^4(I,W^{s-\frac{1}{2}+\frac{3}{20},\infty}(\R))}\leq\F(\widetilde\Xi_k)\lp\lA F_j\rA_{L^4(I,H^{s-\frac{9}{10}}(\R))}+\lA u_j\rA_{L^\infty(I,H^s(\R))}\rp,
		      \end{equation*}
		\item if~$d\ge 2$,
		       \begin{equation*}
		      	\lA u_j\rA_{L^2(I,W^{s-\frac{d}{2}+\frac{3}{10}-\eps,\infty}(\R^d))}\leq\F_{\eps}(\widetilde\Xi_k)\lp\lA F_j\rA_{L^2(I,H^{s-\frac{9}{10}}(\R^d))}+\lA u_j\rA_{L^\infty(I,H^s(\R^d))}\rp,
		       \end{equation*}
	\end{itemize}
	for~$j\geq j_0$.	
\end{cor}
\begin{proof}
 By \eqref{Fj'}, \eqref{eq:regedp} and \eqref{eq:remreg} we have that if $u_j$ is a solution of \eqref{eq:nonregu} then $u_j$ is also a solution of \eqref{eq:regu} with 
\begin{align*}
F_{j\delta}&=F_j+R_j+iR'_j+\lp S_{(j-3)\delta}\gamma(x,D_x)-S_{j-3}\gamma(x, D_x)\rp\Delta_ju+
	      \lp S_{(j-3)\delta}(V)-S_{j-3}(V)\rp\cdot\nabla\Delta_ju\\
	      &=:F_j+\widetilde F_j.
\end{align*}
From Lemma \ref{lem:equipara} we have that $R_j$ and$~R'_j$ are of order $0$. On the other hand, with~$p=4$ if~$d=1$ and~$p=2$ if~$d\geq2$, there holds
	\begin{align*}
		\lA(S_{(j-3)\delta}(V)-S_{j-3}(V))\cdot\nabla u_j\rA_{L^p(I,H^{s-(1-\delta)}(\R^d))}&\leq  \lA V\rA_{E(J)}\lA u_j\rA_{L^\infty(I,H^s(\R^d))},\\
		\lA(S_{(j-3)\delta}(\gamma)-S_{j-3}(\gamma))u_j\rA_{L^p(I,H^{s-\frac{3}{2}(1-\delta)}(\R^d))}&\leq\mathcal{M}_k(\gamma)(J)\lA u_j\rA_{L^\infty(I,H^s(\R^d))}.
	\end{align*}
	Those are classical regularization results (See e.g. \cite{TaylorPseudorNLPDE}, Section~1.3). \\
We deduce that 
\[
\lA \widetilde F_j\rA_{L^p(I,H^{s-\frac{3}{2}(1-\delta)}(\R^d))}\leq\F(\Xi_k)\lA u_j\rA_{L^\infty(I,H^s(\R^d))}.
\]
Now, choosing $\delta=\frac{2}{5}$ gives
$
\varsigma=\frac{3}{2}(1-\delta)
$
and then Corollary \ref{cor:Strfull} concludes the proof.
\end{proof}
Finally,  we prove our main theorem.\\
{\bf Proof of Theorem \ref{main:theo}}\\
Let $u$ be as in the statement of the theorem, by \eqref{ww:reduce} $u$ is a solution of 
\[
		\partial_tu+iT_\gamma u+ T_V\cdot\nabla u=f-iT_\omega u.
\]

Then, by \eqref{eq:sym} ~$\Delta_ju$ solves
\begin{equation*}
	\lp\partial_t+iT_\gamma +T_V\cdot\nabla \rp\Delta_ju=F_j,
\end{equation*}
with
\begin{equation*}
	F_j:=\Delta_j f-
	      i\Delta_j (T_\omega u)+i\lb T_\gamma,\Delta_j\rb u+
+\lb T_V,\Delta_j\rb\cdot\nabla u.
\end{equation*}
	Notice that~$F_j$ has spectrum in~$\C_j$. Applying the symbolic calculus Theorem \ref{theo:sc} we deduce that 
	\begin{align*} 
             \lA \Delta_j (T_{\omega_1} u)\rA_{L^p(I;H^{s-\mez}))}&\leq C\mathcal{N}_k(\omega_1)\lA u\rA_{L^\infty(I;H^s))}\\
		\lA\lb T_\gamma,\Delta_j\rb u\rA_{L^p(I;H^{s-\mez})}&\leq C\mathcal{N}_k(\gamma)\lA u\rA_{L^\infty(I;H^s))},\\
		\lA\lb T_V,\Delta_j\rb\cdot\nabla u\rA_{L^p(I;H^s))}&\leq C\lA V\rA_E\lA u\rA_{L^\infty(I;H^s))}.
	\end{align*}
	Then we can use Corollary~\ref{cor:Strloc} on~$\Delta_ju$ to prove
	 \begin{equation*}
		\lA\Delta_j u\rA_{L^p(I,W^{s-\frac{d}{2}+\mu,\infty}(\R^d))}\leq\F\lp\mathcal{N}_{k}(\gamma)+\lA V\rA_{E}\rp\lp\lA f\rA_{L^p(I,H^{s-\frac{9}{10}}(\R^d))}+\lA u\rA_{L^\infty(I,H^s(\R^d))}\rp
	\end{equation*}	
	for~$j\geq j_0$, and using the bound
	\begin{equation*}
		\lA\Delta_j u\rA_{L^p(I,W^{s-\frac{d}{2}+\mu,\infty}(\R^d))}\leq C2^{j\mu}\lA u\rA_{L^\infty(I,H^s(\R^d))}\leq C\lA u\rA_{L^\infty(I,H^s(\R^d))}
	\end{equation*}
	for~$j<j_0$, we finally obtain
\[
\lA u\rA_{L^p(I,W^{s-\frac{d}{2}+\mu-\eps,\infty}(\R^d))}\le \sum_j 2^{-j\eps} \lA\Delta_j u\rA_{L^p(I,W^{s-\frac{d}{2}+\mu,\infty}(\R^d))}
\]
which is bounded by the desired quantity.

\section{Cauchy problem} \label{sec:Cauchy}
We are now in position to derive the Cauchy theory announced in Theorem~\ref{theo:Cauchy}. Let
\[
(\eta_0, \psi_0)\in H^{s+\mez}\times H^s,\quad s>2-\frac d2+\mu
\]
be the initial data such that $\dist(\eta_0, \Gamma)>h>0$. We regularize $(\eta_0, \psi_0)$ to a sequence $(\eta_0^\eps, \psi_0^\eps)\in H^\infty\times H^\infty$ converging to $(\eta_0, \psi_0)$ in $H^{s+\mez}\times H^s$. Then we can choose a uniform $h_0>0$ such that $\dist(\eta_0^\eps, \Gamma)>h_0$. For each initial condition
$(\eta_0^\eps, \psi_0^\eps)$ we know from the local  well-posedness theory in \cite{ABZ1} that there exists a smooth solution $(\eta^\eps, \psi^\eps)$  to \eqref{eq:Zak} with the maximal life time interval $[0, T^*_\eps)$. Applying our a priori estimate of Proposition~\ref{prop:aprioriestimate} and the Strichartz estimate of Corollary~\ref{cor:Stri} give for each $\eps>0$
\[  M^\eps_{s, T}+Z^\eps_{r, T}\leq\F_{h_0}\lp M^\eps_{s,0}+T^\delta\F\lp M^\eps_{s, T}+Z^\eps_{r,T}\rp\rp,\quad \forall T\in [0, T^*_\eps),~\forall \eps>0.
\]
with obvious notations. Combining this estimate with the blow-up criterion in Proposition \ref{theo:blowup} one deduces by standard argument that there exists a time $T>0$ uniformly in $\eps>0$ such that $T_\eps^*>T$. Set $I=[0, T]$. By virtue of Proposition \ref{theo:contraction}, the sequence $(\eta^\eps, \psi^\eps)$ is Cauchy in 
\[
X^{s-\tmez, r-\tmez}:=C^0(I; H^{s-1}(\xR^d)\times H^{s-\tmez}(\xR^d))\cap  L^p(I; W^{r-\mez}(\xR^d)\times W^{r-1,\infty}(\xR^d))
\]
and therefore converges strongly to some $(\eta, \psi)$ in $X^{s,r}$. On the other hand, this sequence is bounded in
\[
Y^{r,s}:=(\eta, \psi)\in L^\infty(I; H^{s+\mez}(\xR^d)\times H^s(\xR^d))\cap L^p(I; W^{r+\mez}(\xR^d)\times W^{r,\infty}(\xR^d)).
\]
Therefore, it converges strongly to $(\eta, \psi)$ in $X^{s',r'}$ for any $s'<s,~r'<r$ and weakly to $(\eta, \psi)$ in $Y^{r,s}$. Consequently, one can pass to the limit in the system \eqref{eq:Zak} as $\eps\to 0$ to have that  $(\eta, \psi)$ is a distributional solution to \eqref{eq:Zak}. Here, we remark that the only nontrivial point is to pass to the limit in the Dirichlet-Neumann operator $G(\eta_\eps)\psi_\eps$, this is done for example in \cite{ABZ}, Corollary $5.16$. Finally, by interpolation it holds that 
\[
(\eta, \psi) \in  C^0(I; H^{s'+\mez}(\xR^d)\times H^{s'}(\xR^d)),\quad \forall s'<s,
\]
which completes the proof of Theorem \ref{theo:Cauchy}.

\section{Appendix}	\label{appendix}
\label{ap:para}
\begin{defn}\label{Paley}
1. (Littlewood-Paley decomposition) Let~$\psi\in C^\infty_0({\mathbf{R}}^d)$ be such that
$$
\psi(\theta)=1\quad \text{for }\la \theta\ra\le 1,\qquad 
\psi(\theta)=0\quad \text{for }\la\theta\ra>2.
$$
Then we define
\begin{equation*}
\psi_k(\theta)=\psi(2^{-k}\theta)\quad\text{for }k\in \xZ,
\qquad \varphi_0=\psi_0,\quad\text{ and } 
\quad \varphi_k=\psi_k-\psi_{k-1} \quad\text{for }k\ge 1.
\end{equation*}
Given a temperate distribution $u$ and an integer $k$ in $\xN$ we also introduce $S_k u$ 
and $\Delta_k u$ by 
$S_k u=\psi_k(D_x)u$ and $\Delta_k u=S_k u-S_{k-1}u$ for $k\ge 1$ and $\Delta_0u=S_0u$. Then we have the formal decomposition 
$$
u=\sum_{k=0}^{\infty}\Delta_k u.
$$
2. (Zygmund spaces) If~$s$ is any real number, we define the Zygmund class~$C^{s}_*({\mathbf{R}}^d)$ as the 
space of tempered distributions~$u$ such that
$$
\lA u\rA_{C^{s}_*}:= \sup_{q\ge 0} 2^{qs}\lA \Delta_q u\rA_{L^\infty}<+\infty.
$$
3. (H\"older spaces) For~$k\in\xN$, we denote by $W^{k,\infty}({\mathbf{R}}^d)$ the usual Sobolev spaces.
For $\rho= k + \sigma$, $k\in \xN, \sigma \in (0,1)$ denote 
by~$W^{\rho,\infty}({\mathbf{R}}^d)$ 
the space of functions whose derivatives up to order~$k$ are bounded and uniformly H\"older continuous with 
exponent~$\sigma$. 
\end{defn}
\begin{rem}\label{rem:Holder}
When $s\in (0,\infty)\setminus \xN$ we have
\[
C_*^s(\xR^d)\equiv W^{s,\infty}(\xR^d).
\]
For $n\in \xN$ we still have the following estimate
\[
\sup_{q\ge 0}2^{nq}\lA \Delta_q u\rA_{L^\infty}\le C(n)\lA u\rA_{W^{n,\infty}}.
\]
\end{rem}
\begin{defn}\label{para:sym+ope}
1. (Symbols) Given~$\rho\in [0, \infty)$ and~$m\in\xR$,~$\Gamma_{\rho}^{m}({\mathbf{R}}^d)$ denotes the space of
locally bounded functions~$a(x,\xi)$
on~${\mathbf{R}}^d\times({\mathbf{R}}^d\setminus 0)$,
which are~$C^\infty$ with respect to~$\xi$ for~$\xi\neq 0$ and
such that, for all~$\alpha\in\xN^d$ and all~$\xi\neq 0$, the function
$x\mapsto \partial_\xi^\alpha a(x,\xi)$ belongs to~$W^{\rho,\infty}({\mathbf{R}}^d)$ and there exists a constant
$C_\alpha$ such that,
\begin{equation}\label{para:symbol}
\forall\la \xi\ra\ge \mez,\quad 
\lA \partial_\xi^\alpha a(\cdot,\xi)\rA_{W^{\rho,\infty}(\xR^d)}\le C_\alpha
(1+\la\xi\ra)^{m-\la\alpha\ra}.
\end{equation}
Let $a\in \Gamma_{\rho}^{m}({\mathbf{R}}^d)$, we define the semi-norm
\begin{equation}\label{defi:semi-norm}
M_{\rho}^{m}(a)= 
\sup_{\la\alpha\ra\le d/2+1+\rho ~}\sup_{\la\xi\ra \ge 1/2~}
\lA (1+\la\xi\ra)^{\la\alpha\ra-m}\partial_\xi^\alpha a(\cdot,\xi)\rA_{W^{\rho,\infty}({\mathbf{R}}^d)}.
\end{equation}
2. (Paradifferential operators) Given a symbol~$a$, we define
the paradifferential operator~$T_a$ by
\begin{equation}\label{eq.para}
\widehat{T_a u}(\xi)=(2\pi)^{-d}\int \chi(\xi-\eta,\eta)\widehat{a}(\xi-\eta,\eta)\varrho(\eta)\widehat{u}(\eta)
\, d\eta,
\end{equation}
where
$\widehat{a}(\theta,\xi)=\int e^{-ix\cdot\theta}a(x,\xi)\, dx$
is the Fourier transform of~$a$ with respect to the first variable; 
$\chi(\theta,\eta)$  is defined by
\bq\label{chi}
\chi(\theta,\eta)=\sum_{k=0}^{+\infty} \psi_{k-3}(\theta) \varphi_k(\eta);
\eq
and~$\varrho\in C^\infty(\R^d),~\varrho=0$ if $|\xi|\le \mez$ and $\varrho=1$ if $|\xi|\ge 1$.
\end{defn}
\begin{rem}
The cut-off function $\chi$ has the following properties for some $0<\eps_1<\eps_2<1$
\bq\label{chi:prop}
\begin{cases}
\chi(\eta, \xi)=1,& \text{for}~|\eta|\le \eps_1(1+|\xi),\\
\chi(\eta, \xi)=0,&\text{for}~|\eta|\ge \eps_2(1+|\xi).
\end{cases}
\eq
\end{rem}
Symbolic calculus for paradifferential operators is summarized in the following theorem (see \cite{MePise}, \cite{Bony}).
\begin{thm}\label{theo:sc}(Symbolic calculus)
Let~$m\in\xR$ and~$\rho\in [0, \infty)$. \\
$(i)$ If~$a \in \Gamma^m_0({\mathbf{R}}^d)$, then~$T_a$ is of order~$\le m$. 
Moreover, for all~$\mu\in\xR$ there exists a constant~$K$ such that
\begin{equation}\label{esti:quant1}\lA T_a \rA_{H^{\mu}\rightarrow H^{\mu-m}}\le K M_{0}^{m}(a).
\end{equation}
$(ii)$ If~$a\in \Gamma^{m}_{\rho}({\mathbf{R}}^d), b\in \Gamma^{m'}_{\rho}({\mathbf{R}}^d)$ with $\rho>0$. Then 
$T_a T_b -T_{a \sharp b}$ is of order~$\le m+m'-\rho$ where
\[
a\sharp b:=\sum_{|\alpha|<\rho}\frac{(-i)^{\alpha}}{\alpha !}\partial_{\xi}^{\alpha}a(x, \xi)\partial_x^{\alpha}b(x, \xi).
\] 
Moreover, for all~$\mu\in\xR$ there exists a constant~$K$ such that
\begin{equation}\label{esti:quant2}
\lA T_a T_b  - T_{a \sharp b}   \rA_{H^{\mu}\rightarrow H^{\mu-m-m'+\rho}}
\le 
K M_{\rho}^{m}(a)M_{0}^{m'}(b)+K M_{0}^{m}(a)M_{\rho}^{m'}(b).
\end{equation}
$(iii)$ Let~$a\in \Gamma^{m}_{\rho}({\mathbf{R}}^d)$ with $\rho >0$. Denote by 
$(T_a)^*$ the adjoint operator of~$T_a$ and by~$\overline{a}$ the complex conjugate of~$a$. Then 
$(T_a)^* -T_{a^*}$ is of order~$\le m-\rho$ where
\[
a^*=\sum_{|\alpha|<\rho}\frac{1}{i^{|\alpha|}\alpha!}\partial_{\xi}^{\alpha}\partial_x^{\alpha}\overline{a}.
\]
Moreover, for all~$\mu$ there exists a constant~$K$ such that
\begin{equation}\label{esti:quant3}
\lA (T_a)^*   - T_{\overline{a}}   \rA_{H^{\mu}\rightarrow H^{\mu-m+\rho}}\le 
K M_{\rho}^{m}(a).
\end{equation}
\end{thm}
\bibliographystyle{plain}

\begin{thebibliography}{1}

\bibitem{ABZ}
Thomas {Alazard}, Nicolas {Burq}, and Claude {Zuily}.
\newblock {Strichartz estimates and the Cauchy problem for the gravity water
  waves equations}.
\newblock {\em  arXiv:1404.4276}, April 2014.

\bibitem{ABZ1}
Thomas Alazard, Nicolas Burq, and Claude Zuily.
\newblock On the water waves equations with surface tension.
\newblock {\em Duke Math. J.}, 158(3):413--499, 2011.

\bibitem{ABZ2}
Thomas Alazard, Nicolas Burq, and Claude Zuily.
\newblock Strichartz estimates for water waves.
\newblock {\em Ann. Sci. {\'E}c. Norm. Sup{\'e}r. (4)}, 44(5):855--903, 2011.

\bibitem{ABZ3}
Thomas Alazard, Nicolas Burq, and Claude Zuily.
\newblock On the Cauchy problem for gravity water waves.
\newblock {\em Invent.Math.}, 198(1): 71--163, 2014.

\bibitem{AD1}
Thomas Alazard, Jean-Marc Delort.
\newblock Global solutions and asymptotic behavior for two dimensional gravity water waves. 
\newblock {\em  arXiv:1305.4090}, 2013.

\bibitem{AD2}
Thomas Alazard, Jean-Marc Delort.
\newblock Sobolev estimates for two dimensional gravity water waves. 
\newblock {\em   arXiv:1307.3836}, 2013.

\bibitem{AM1}
David M. Ambrose, Nader Masmoudi.
\newblock The zero surface tension limit of two-dimensional water waves.
\newblock {\em Comm. Pure Appl. Math.} 58 (2005), no. 10, 1287–1315. 

\bibitem{AM2}
David M. Ambrose, Nader Masmoudi.
\newblock The zero surface tension limit of three-dimensional water waves.
\newblock {\em Indiana Univ. Math. J.} 58 (2009), no. 2, 479–521. 


\bibitem{BG}
Klaus Beyer, Matthias G\"{u}nther
\newblock On the Cauchy problem for a capillary drop. I. Irrotational motion.
\newblock  {\em Math. Methods Appl. Sci.} 21 (1998), no. 12, 1149–1183. 

\bibitem{BaCh-IMRN}
Hajer Bahouri and Jean-Yves Chemin.
\newblock \'{E}quations d'ondes quasilin\'eaires et effet dispersif.
\newblock {\em Internat. Math. Res. Notices}, (21):1141--1178, 1999.

\bibitem{BaCh-AJM}
Hajer Bahouri and Jean-Yves Chemin.
\newblock \'{E}quations d'ondes quasilin\'eaires et estimations de
  {S}trichartz.
\newblock {\em Amer. J. Math.}, 121(6):1337--1377, 1999.

\bibitem{Bony}
Jean-Michel Bony.
\newblock Calcul symbolique et propagation des singularit\'es pour les
  \'equations aux d\'eriv\'ees partielles non lin\'eaires.
\newblock {\em Ann. Sci. \'Ecole Norm. Sup. (4)}, 14(2):209--246, 1981.

\bibitem{BGT}
Nicolas Burq, Patrick G\'erard, and Nikolay Tzvetkov.
\newblock Strichartz inequalities and the nonlinear {S}chr\"odinger equation on
  compact manifolds.
\newblock {\em Amer. J. Math.}, 126(3):569--605, 2004.

\bibitem{CHS}
Hans Christianson, Vera~Mikyoung Hur, and Gigliola Staffilani.
\newblock Strichartz estimates for the water-wave problem with surface tension.
\newblock {\em Comm. Partial Differential Equations}, 35(12):2195--2252, 2010.

\bibitem{CouShk}
Daniel Coutand and Steve Shkoller.
\newblock Well-posedness of the free-surface incompressible Euler equations with or without surface tension. 
\newblock {\em J. Amer. Math. Soc. }20(3): 829–930, 2007.

\bibitem{Craig}
Walter Craig.
\newblock An existence theory for water waves and the Boussinesq and Korteweg-de Vries scaling limits.
\newblock {\em Comm. Partial Differential Equations} 10 (1985), no. 8, 787–1003.

\bibitem{CrSu}
Walter Craig and Catherine Sulem.
\newblock Numerical simulation of gravity waves. 
\newblock {\em J. Comput. Phys. }108(1): 73–83, 1993.

\bibitem{CL}
Demetrios Christodoulou, Hans Lindblad.
\newblock On the motion of the free surface of a liquid.
\newblock {\em  Comm. Pure Appl. Math.}  53 (2000), no. 12, 1536–1602.

\bibitem{GMS1}
Pierre Germain, Nader Masmoudi, Jalal Shatah. 
\newblock Global solutions for the gravity water waves equation in dimension 3.
\newblock {\em Ann. of Math.} (2) 175 (2012), no. 2, 691–754. 

\bibitem{GMS2}
Pierre Germain, Nader Masmoudi, Jalal Shatah. 
\newblock Global existence for capillary water waves.
\newblock {\em  Comm. Pure Appl. Math.} 68 (2015), no. 4, 625–687.

\bibitem{HIT}
 John Hunter, Mihaela Ifrim, Daniel Tataru.
\newblock Two dimensional water waves in holomorphic coordinates.
\newblock{\em  arXiv:1401.1252}, 2014.

\bibitem{IT1}
Mihaela Ifrim, Daniel Tataru.
\newblock Two dimensional water waves in holomorphic coordinates II: global solutions. 
\newblock{\em   arXiv:1404.7583}, 2014.

\bibitem{IT2}
Mihaela Ifrim, Daniel Tataru.
\newblock The lifespan of small data solutions in two dimensional capillary water waves. 
\newblock{\em arXiv:1406.5471}, 2014.

\bibitem{IP1}
Alexandru D. Ionescu, Fabio Pusateri.
\newblock Global solutions for the gravity water waves system in 2d. 
\newblock {\em Invent. Math.} 199 (2015), no. 3, 653–804. 

\bibitem{IP2}
Alexandru D. Ionescu, Fabio Pusateri.
\newblock Global regularity for 2d water waves with surface tension. 
\newblock {\em  arXiv:1408.4428}, 2014.

\bibitem{Lannes}
David Lannes.
\newblock Well-posedness of the water-waves equations.
\newblock {\em J. Amer. Math. Soc.} 18 (2005), no. 3, 605–654 (electronic). 

\bibitem{Gill}
Gilles Lebeau.
\newblock Controle de l'équation de Schrodinger.
\newblock{\em J. Math. Pures Appl.}, (9) 71, 267--291, 1992.

\bibitem{Nalimov}
V.I. Nalimov.
\newblock The Cauchy-Poisson problem.
\newblock  {\em Dinamika Splošn. Sredy Vyp. 18 Dinamika Židkost. so Svobod. Granicami} (1974), 104–210, 254.

\bibitem{NgPo}
Quang Huy Nguyen and Thibault de Poyferr\'{e}.
\newblock A paradifferential reduction for the gravity-capillary waves system at low regularity and applications.

\bibitem{Lind}
Hans Lindblad.
\newblock Well-posedness for the motion of an incompressible liquid with free surface boundary.
\newblock {\em Ann. of Math.} (2) 162 (2005), no. 1, 109–194.

\bibitem{MePise}
Guy M{\'e}tivier.
\newblock {\em Para-differential calculus and applications to the {C}auchy
  problem for nonlinear systems}, volume~5 of {\em Centro di Ricerca Matematica
  Ennio De Giorgi (CRM) Series}.
\newblock Edizioni della Normale, Pisa, 2008.

\bibitem{MiZh}
Mei Ming and Zhifei Zhang.
\newblock Well-posedness of the water-wave problem with surface tension.
\newblock{\em J. Math. Pures Appl.} (9) 92, no. 5, 429--455, 2009. 

\bibitem{SZ1}
Jalal Shatah and Chongchun Zeng.
\newblock Geometry and a priori estimates for free boundary problems of the
  {E}uler equation.
\newblock {\em Comm. Pure Appl. Math.}, 61(5):698--744, 2008.

\bibitem{SZ2}
Jalal Shatah and Chongchun Zeng.
\newblock  A priori estimates for fluid interface problems.
\newblock{\em Comm. Pure Appl. Math.} 61(6): 848–876, 2008.

\bibitem{SZ3}
Jalal Shatah and Chongchun Zeng.
 \newblock Local well-posedness for fluid interface problems. 
\newblock{\em Arch. Ration. Mech. Anal.} 199(2): 653–705, 2011.

\bibitem{StTa}
Gigliola Staffilani and Daniel Tataru.
\newblock Strichartz estimates for a {S}chr\"odinger operator with nonsmooth
  coefficients.
\newblock {\em Comm. Partial Differential Equations}, 27(7-8):1337--1372, 2002.

\bibitem{TataruNS}
Daniel Tataru.
\newblock Strichartz estimates for operators with nonsmooth coefficients and
  the nonlinear wave equation.
\newblock {\em Amer. J. Math.}, 122(2):349--376, 2000.

\bibitem{TaylorPseudorNLPDE}
Michael~E. Taylor.
\newblock {\em Pseudodifferential operators and nonlinear {PDE}}, volume 100 of
  {\em Progress in Mathematics}.
\newblock Birkh\"auser Boston, Inc., Boston, MA, 1991.

\bibitem{Wu1}
Sijue Wu.
\newblock Well-posedness in Sobolev spaces of the full water wave problem in 2-D.
\newblock {\em Invent. Math. 130 (1997)}, no. 1, 39–72.

\bibitem{Wu2}
Sijue Wu.
\newblock Well-posedness in Sobolev spaces of the full water wave problem in 3-D.
\newblock {\em J. Amer. Math. Soc.} 12 (1999), no. 2, 445–495.

\bibitem{Wu2D}
Sijue Wu.
\newblock Almost global wellposedness of the 2-D full water wave problem.
\newblock {\em Invent. Math.} 177 (2009), no. 1, 45–135. 

\bibitem{Wu3D}
Sijue Wu.
\newblock Global wellposedness of the 3-D full water wave problem.
\newblock {\em Invent. Math.} 184 (2011), no. 1, 125–220. 

\bibitem{Yosihara}
Hideaki Yosihara.
\newblock Gravity waves on the free surface of an incompressible perfect fluid
  of finite depth.
\newblock {\em Publ. Res. Inst. Math. Sci.}, 18(1):49--96, 1982.

\bibitem{Zack}
Vladimir Zakharov.
\newblock Stability of periodic waves of finite amplitude on the surface of a
  deep fluid.
\newblock {\em Journal of Applied Mechanics and Technical Physics}, 1968.


\bibitem{ZworskiSemClass}
Maciej Zworski.
\newblock {\em Semiclassical analysis}, volume 138 of {\em Graduate Studies in
  Mathematics}.
\newblock American Mathematical Society, Providence, RI, 2012.

\end{thebibliography}

\end{document}